\documentclass{article}

\usepackage{arxiv}

\usepackage[utf8]{inputenc} 
\usepackage[T1]{fontenc}    
\usepackage{url}            
\usepackage{booktabs}       
\usepackage{amsfonts}       
\usepackage{nicefrac}       
\usepackage{microtype}      
\usepackage{lipsum}

\RequirePackage[OT1]{fontenc}
\RequirePackage{amsthm,amsmath}
\RequirePackage[numbers]{natbib}
\RequirePackage[colorlinks,citecolor=blue,urlcolor=blue]{hyperref}
\hypersetup{colorlinks=true,linkcolor=blue, linktocpage}

\usepackage[utf8]{inputenc}
\usepackage[english]{babel}
\usepackage{amsmath}
\usepackage{amsfonts}
\usepackage{amssymb}
\usepackage{caption}
\usepackage{array}
\usepackage{framed}
\usepackage{multirow}
\usepackage{bbm}
\usepackage{xcolor}
\usepackage{mathrsfs}
\usepackage{pdfsync,xargs}
\usepackage{enumitem}
\usepackage[textwidth=80pt]{todonotes}
\usepackage{aliascnt}
\usepackage{cleveref}

\usepackage[ruled]{algorithm2e}
\usepackage{appendix}

\newtheorem{thm}{Theorem}

\crefname{theorem}{Theorem}{Theorems}
\Crefname{theorem}{Theorem}{Theorems}
\crefname{corollary}{Corollary}{Corollaries}
\Crefname{corollary}{Corollary}{Corollaries}
\Crefname{chapter}{Chapter}{Chapters}
\crefname{chapter}{Chapter}{Chapters}

\newtheorem{ex}{Example}
\crefname{example}{example}{examples}
\Crefname{example}{Example}{Examples}

\newaliascnt{lemma}{index}
\newtheorem{lemma}[lemma]{Lemma}
\aliascntresetthe{lemma}
\crefname{lemma}{Lemma}{lemmas}
\Crefname{Lemma}{Lemma}{Lemmas}

\newaliascnt{proposition}{index}
\newtheorem{prop}[proposition]{Proposition}
\aliascntresetthe{proposition}
\crefname{proposition}{Proposition}{Propositions}
\Crefname{Proposition}{Proposition}{Propositions}

\newaliascnt{corollary}{index}

\aliascntresetthe{corollary}
\crefname{corollary}{corollary}{corollaries}
\Crefname{Corollary}{Corollary}{Corollaries}

\newaliascnt{definition}{index}
\newtheorem{defi}[definition]{Definition}
\aliascntresetthe{definition}
\crefname{definition}{definition}{definitions}
\Crefname{Definition}{Definition}{Definitions}

\newaliascnt{remark}{index}
\newtheorem{rem}[remark]{Remark}
\aliascntresetthe{remark}
\crefname{remark}{remark}{remarks}
\Crefname{remark}{Remark}{Remarks}

\newenvironment{enumerateList}{ \begin{enumerate}[label=(\roman*),wide=0pt, labelindent=\parindent]}{\end{enumerate}}

\newtheoremstyle{styledef}
  {6pt}
  {6pt}
  {\sffamily}
  {0em}
  {\bfseries}
  {}
  {1.5em}
  {}

\theoremstyle{styledef}
\newenvironment{hyp}[1]{
\begin{enumerate}[label=\textbf{\sf(#1\arabic*)},resume=hyp#1]\begin{sf}}
{\end{sf}\end{enumerate}}
\crefname{hyp}{}{ass}
\Crefname{hyp}{}{Ass}

\newcommandx{\admiss}[1][1=f]{\mathsf{A}_{#1}}
\newcommand{\arginf}{\mathrm{arginf}}

\newcommand{\as}{\mathrm{a.s.}}

\newcommandx{\binfty}[1][1=\alpha]{|b|_{\infty, #1}}
\newcommandx{\balpha}[1][1=\alpha]{|b|_{#1}}
\newcommandx{\bmuf}[2][1=\mu, 2=\alpha]{ {b_{#1, #2}}}
\newcommandx{\bmufk}[2][1=\mu, 2=\alpha]{\hat{b}_{#1, #2, M}}

\newcommandx{\couple}[2][1=\PQ, 2=\PP]{(#1 || #2)}
\newcommand{\Cov}{\mathbb{C}\mathrm{ov}}
\newcommand{\cte}{\kappa}
\newcommandx{\cteLipSto}[2][1=N, 2=M]{|b|_{\hmu_{#1}, #2, \alpha}}
\newcommandx{\cteLipStoInf}[1][1=\alpha]{B_{#1}}
\newcommandx{\ctemono}[1][1=\alpha]{L_{#1, 1}}
\newcommandx{\cteinf}[1][1=\alpha]{L_{#1, 2}}
\newcommandx{\ctesup}[1][1=\alpha]{L_{#1, 3}}
\newcommand{\data}{\mathscr{D}}

\newcommand{\dlim}{\Rightarrow}

\newcommandx{\diverg}[1][1=\alpha]{D_{#1}}
\newcommandx{\Domain}[1][1=\alpha]{\mathrm{Dom}_{#1}}

\newcommand{\eqdef}{:=}

\newcommand{\eqsp}{\;}

\newcommandx{\falpha}[1][1=\alpha]{f_{#1}}
\newcommandx{\aei}[1][1=\alpha]{$(#1, \Gamma)$-}
\newcommandx{\GammaAlpha}[1][1=\alpha]{ \Gamma}
\renewcommand{\geq}{\geqslant}
\newcommandx{\gmuf}[1][1=\mu]{ {g_{#1}}}
\newcommand{\hmu}{\hat{\mu}}
\newcommandx{\hbinfty}[1][1=\alpha]{|\hat{b}|_{\infty, #1}}

\newcommand{\indi}[1]{\mathbf{1}_{#1}}
\newcommand{\indiacc}[1]{\mathbf{1}_{\{#1\}}}

\newcommandx{\iteration}[1][1=\alpha]{\mathcal{I}_{#1}}
\newcommandx{\iterationK}[1][1=\alpha]{\hat{\mathcal{I}}_{#1, M}}

\newcommandx{\lbd}[2][1=]{
    \ifthenelse{\equal{#1}{}}
    {{\boldsymbol{\lambda}_{#2}}}
    {\lambda_{#1,#2}}
    }

\renewcommand{\leq}{\leqslant}

\newcommand{\llim}{\lim \limits}
\newcommand{\lliminf}{\liminf \limits}
\newcommand{\llimsup}{\limsup \limits}

\newcommand{\lr}[1]{\left(#1 \right)}
\newcommand{\lrb}[1]{\left[#1 \right]}
\newcommand{\lrc}[1]{\left\{#1 \right\}}
\newcommand{\lrcb}[1]{\left\{#1 \right\}}

\newcommand{\mcf}{\mathcal{F}}
\newcommand{\meas}[1]{\mathrm{M}_{#1}}
\newcommand{\muf}{{\mu^\star}}
\newcommand{\mubar}{{\bar{\mu}}}
\newcommand{\nset}{\mathbb N}
\newcommand{\nstar}{\mathbb{N}^\star}
\newcommand{\PE}{\mathbb E}
\newcommandx{\posterior}[1][1=y]{p(#1|\mathscr{D})}
\newcommand{\PP}{\mathbb P}
\newcommand{\PQ}{\mathbb Q}
\newcommandx{\Psifn}[1][1=\alpha]{\Psi_{#1, L}}
\newcommandx{\Psif}[1][1=\alpha]{\Psi_{#1}}
\newcommandx{\Psifk}[3][1=\alpha, 2=N, 3=M]{\hat{\Psi}_{#1, #3, #2}}
\renewcommand{\rho}{\varrho}
\newcommand{\rmd}{\mathrm d}

\newcommand{\Rset}{\mathbb{R}}
\newcommand{\rset}{\mathbb{R}}

\newcommand{\set}[2]{\lrc{#1\eqsp: \eqsp #2}}

\newcommand{\simplex}{\mathcal{S}}
\newcommand{\tgamma}{{\tilde{\Gamma}}}
\newcommandx{\thetat}[2][1=j, 2=t]{\theta_{#1,#2}}

\newcommand{\tv}[1]{\left\|#1\right\|_{TV}}

\newcommand{\Tset}{{\mathsf{T}}}
\newcommand{\Tsigma}{\mathcal{T}}

\newcommand{\Yset}{\mathsf Y}
\newcommand{\Ysigma}{\mathcal Y}
\newcommand{\Var}{\mathbb{V}\mathrm{ar}}

\begin{document}

\title{Infinite-dimensional gradient-based descent for alpha-divergence minimisation}

\author{
  Kam\'elia~Daudel \\
LTCI, Télécom Paris \\
Institut Polytechnique de Paris, France \\
\texttt{kamelia.daudel@telecom-paris.fr} \\
 \And
Randal~Douc \\
SAMOVAR, Télécom SudParis \\
Institut Polytechnique de Paris, France \\
\texttt{randal.douc@telecom-sudparis.eu} \\
\AND
François Portier \\
LTCI, Télécom Paris \\
Institut Polytechnique de Paris, France \\
\texttt{francois.portier@telecom-paris.fr} \\
}

\maketitle

\def\abstractname{}
\begin{abstract}
This paper introduces the \aei descent, an iterative algorithm which operates on measures and performs $\alpha$-divergence minimisation in a Bayesian framework. This gradient-based procedure extends the commonly-used variational approximation by adding a prior on the variational parameters in the form of a measure. We prove that for a rich family of functions $\GammaAlpha$, this algorithm leads at each step to a systematic decrease in the $\alpha$-divergence and derive convergence results. Our framework recovers the Entropic Mirror Descent algorithm and provides an alternative algorithm that we call the Power Descent. Moreover, in its stochastic formulation, the \aei descent allows to optimise the mixture weights of any given mixture model without any information on the underlying distribution of the variational parameters. This renders our method compatible with many choices of parameters updates and applicable to a wide range of Machine Learning tasks. We demonstrate empirically on both toy and real-world examples the benefit of using the Power descent and going beyond the Entropic Mirror Descent framework, which fails as the dimension grows.
\end{abstract}



\section{Introduction}

Bayesian statistics for complex models often induce intractable and hard-to-compute posterior densities which need to be approximated. Variational methods such as Variational Inference (VI) \cite{Jordan1999, beal.phd} and Expectation Propagation (EP) \cite{Opper:2000:GPC:1121900.1121911, Minka:2001:EPA:2074022.2074067} consider this objective purely as an optimisation problem (which is often non-convex). These approaches seek to approximate the posterior density by a simpler variational density $k_\theta$, characterized by a set of variational parameters $\theta \in \Tset$, where $\Tset$ is the parameter space. In these methods $\theta$ is optimised such that it minimizes a certain objective function, typically the Kullback-Leibler divergence \cite{kullback1951} between the posterior and the variational density.

Modern Variational methods improved in three major directions \cite{2016arXiv160100670B, 2017arXiv171105597Z} (i) Black-Box inference techniques \cite{2012arXiv1206.6430P, 2014arXiv1401.0118R} and Hierarchical Variational Inference methods \cite{pmlr-v48-ranganath16, pmlr-v80-yin18b} have been deployed, expanding the variational family and rendering Variational methods applicable to a wide range of models (ii) Algorithms based on alternative families of divergences such as the $\alpha$-divergence \cite{zhu-rohwer-alpha-div, ZhuRohwer} and Renyi's $\alpha$-divergence \cite{renyi1961,2012arXiv1206.2459V} have been introduced \cite{divergence-measures-and-message-passing, power-ep, 2015arXiv151103243H, 2016arXiv160202311L, NIPS2017_6866, NIPS2017_7093, NIPS2018_7816} to bypass practical issues linked to the Kullback-Leibler divergence \cite{Minka:2001:EPA:2074022.2074067, 2016arXiv160100670B, JMLR:v14:hoffman13a} (iii) Scalable methods relying on stochastic optimisation techniques \cite{10.1007/978-3-7908-2604-3_16, robbins1951} have been developed to enable large-scale learning and have been applied to complex probabilistic models \cite{JMLR:v14:hoffman13a,sep, 2015arXiv150308060D,LDA2001}.

In the spirit of Hierarchical Variational Inference, we offer in this paper to enlarge the variational family by adding a prior on the variational density $k_\theta$ and considering
$$
q(y) = \int_\Tset \mu(\rmd \theta) k_\theta(y) \eqsp,
$$
which is a more general form compared to the one found in \cite{pmlr-v80-yin18b} where $\mu$ is parametrised by another parametric model. As for the objective function, we work within the $\alpha$-divergence family, which admits the forward Kullback-Leibler and the reverse Kullback-Leibler as limiting cases. These divergences belong to the $f$-divergence family \cite{csiszar,morimoto} and as such, they have convexity properties so that the minimisation of the $\alpha$-divergence between the targeted posterior density and the variational density $q$ with respect to $\mu$ can be seen as a convex optimisation problem. \newline

The paper is then organised as follows:
\begin{itemize}[wide=0pt, labelindent=\parindent]
\item In \Cref{section:optim:problem}, we briefly review basic concepts around the $\alpha$-divergence family before recalling the basics of Variational methods and formulating formally the optimisation problem we consider.
\item In \Cref{section:aei}, we describe the Exact \aei descent, an iterative algorithm that performs $\alpha$-divergence minimisation by updating the measure $\mu$. We establish in \Cref{thm:monotone} sufficient conditions on $\GammaAlpha$ for this algorithm to lead at each step to a systematic decrease in the $\alpha$-divergence. We then investigate the convergence of the algorithm in \Cref{thm:admiss}, \ref{thm:admiss:spec} and \ref{thm:repulsive}. Strikingly, the Infinite-dimensional Entropic Mirror Descent \cite[Appendix A]{pmlr-v97-hsieh19b} is included in our framework and we obtain an $O(1/N)$ convergence rate under minimal assumptions, which improves on existing results and illustrates the generality of our approach. We also introduce a novel algorithm called the Power Descent, for which we prove convergence to an optimum and obtain an $O(1/N)$ convergence rate when $\alpha > 1$.
\item In \Cref{sec:sto}, we define the Stochastic version of the Exact \aei descent and apply it to the important case of mixture models \cite{Jaakkola98improvingthe, DBLP:journals/corr/abs-1206-4665}. The resulting general-purpose algorithm is Black-Box and does not require any information on the underlying distribution of the variational parameters. This algorithm notably enjoys an $O(1/\sqrt{N})$ convergence rate in the particular case of the Entropic Mirror Descent if we know the stopping time of the algorithm (\Cref{thm:emd:sto}).
\item Finally, \Cref{section:applications} is devoted to numerical experiments. We demonstrate the benefit of using the Power Descent and thus of going beyond the Entropic Mirror Descent framework. We also compare our method to a computationally equivalent Adaptive Importance Sampling algorithm for Bayesian Logistic Regression on a large dataset.
\end{itemize}

\section{Formulation of the optimisation problem}
\label{section:optim:problem}

\subsection{The $\alpha$-divergence}
Let $(\Yset,\Ysigma, \nu)$ be a measured space, where $\nu$ is a $\sigma$-finite measure on $(\Yset, \Ysigma)$. Let $\PQ$ and $\PP$ be two probability measures on $(\Yset, \Ysigma)$ that are absolutely continuous with respect to $\nu$ i.e. $\PQ \preceq \nu$, $\PP \preceq \nu$. Let us denote by $q = \frac{\rmd \PQ}{\rmd \nu}$ and $p = \frac{\rmd \PP}{\rmd \nu}$ the Radon-Nikodym derivatives of $\PQ$ and $\PP$ with respect to $\nu$.


\begin{defi}\label{def:div} Let $\alpha \in \Rset \setminus \lrcb{0,1}$. The $\alpha$-divergence and the Kullback-Leibler (KL) divergence between $\PQ$ and $\PP$ are respectively defined by :
\begin{align*} 
\diverg \couple[\PQ][\PP] &= \int_\Yset \frac{1}{\alpha(\alpha-1)} \left[\left(\dfrac{q(y)}{p(y)}\right)^{\alpha} -1 \right] p(y) \nu(\rmd y)\eqsp, \\
\diverg[KL] \couple[\PQ][\PP] &= \int_\Yset \log\left(\dfrac{q(y)}{p(y)}\right) q(y) \nu(\rmd y) \eqsp,
\end{align*}
wherever they are well-defined (and otherwise we write $+ \infty$).
\end{defi}



%


As $\lim_{\alpha \to 0} \diverg\couple[\PQ][\PP] = \diverg[KL]\couple[\PP][\PQ]$ and $\lim_{\alpha \to 1} \diverg\couple[\PQ][\PP] = \diverg[KL]\couple[\PQ][\PP]$ (see for example \cite{2012arXiv1206.2459V}), the definition of the $\alpha$-divergence can be extended to $0$ and $1$ by continuity and we will use the notation $\diverg[0] \couple[\PQ][\PP] = \diverg[KL]\couple[\PP][\PQ]$ and $\diverg[1]\couple[\PQ][\PP] = \diverg[KL]\couple[\PQ][\PP]$ throughout the paper. Letting $\falpha$ be the convex function on $(0, +\infty)$ defined by $\falpha[0](u) = u-1 -\log(u)$, $\falpha[1](u) = 1 - u + u\log(u)$ and $\falpha(u) = \frac{1}{\alpha(\alpha-1)} \left[  u^\alpha -1 - \alpha(u-1) \right]$ for all $\alpha \in \rset \setminus \lrcb{0,1}$, we have that for all $\alpha \in \rset$,
\begin{align}\label{eq:gen:div}
\diverg\couple[\PQ][\PP] = \int_\Yset \falpha\left(\frac{q(y)}{p(y)} \right)p(y) \nu(\rmd y) \eqsp.
\end{align}
Written under that form, the r.h.s of \eqref{eq:gen:div} corresponds to the general definition of the $\alpha$-divergence, that is $q$ and $p$ do not need to be normalised in \eqref{eq:gen:div} in order to define a divergence. We next remind the reader of a few more results about the $\alpha$-divergence and we refer to \cite{2012arXiv1206.2459V, alpha-beta-gamma, Cichocki_2011, DBLP:journals/corr/abs-1804-06334} for more details on the $\alpha$-divergence family.

\begin{prop}\label{prop:f-div} The $\alpha$-divergence is always non-negative and it is equal to zero if and only if $\PQ=\PP$. Furthermore, it is jointly convex in $\PQ$ and $\PP$ and for all $\alpha \in \rset$, $\diverg\couple[\PQ][\PP] = \diverg[1-\alpha]\couple[\PP][\PQ]$. 
\end{prop}


Special cases of the $\alpha$-divergence family include the Hellinger distance \cite{Hellinger1909, lindsay1994} and the $\chi^2$-divergence \cite{NIPS2017_6866} which correspond respectively to order $\alpha = 0.5$ and $\alpha = 2$.


\subsection{Variational Inference within the $\alpha$-divergence family}

Assume that we have access to some observed variables $\data$ generated from a probabilistic model $p(\data|y)$ parameterised by a hidden random variable $y \in \Yset$ that is drawn from a certain prior $p_0(y)$. Bayesian inference involves being able to compute or sample from the posterior density of the latent variable $y$ given the data $\data$:
$$
p(y | \data) = \frac{p(y, \data)}{p(\data)} = \frac{ p_0(y) p(\data|y)}{p(\data)} \eqsp,
$$
where $p(\data) = \int_\Yset p_0(y) p(\data|y) \nu(\rmd y)$ is called the {\em marginal likelihood} or {\em model evidence}. For many useful models the posterior density is intractable due to the normalisation constant $p(\data)$. One example of such a model is Bayesian Logistic Regression for binary classification.

\begin{ex}[Bayesian Logistic Regression]\label{ex:BLR} We use the same setting as in \cite{gershman2012nonparametric}. We observe the data $\data = \lrcb{\boldsymbol{c}, \boldsymbol{x}}$ which is made of $I$ binary class labels, $c_i \in \lrcb{-1, 1}$, and of $L$ covariates for each datapoint, $\boldsymbol{x}_i \in \Rset^L$. The hidden
variables $y = \lrcb{\boldsymbol{w}, \beta}$ consist of $L$ regression coefficients $w_l \in \Rset$, and a precision parameter $\beta \in \Rset^+$. We assume the following model
\begin{align*}
  &p_0(\beta) = \mathrm{Gamma}(\beta; a, b) \eqsp , \\
  &p_0(w_l| \beta) = \mathcal{N}(w_l; 0, \beta^{-1}) \eqsp, \quad 1 \leq l \leq L \eqsp, \\
  &p(c_i = 1 | \boldsymbol{x}_i, \boldsymbol{w}) = \frac{1}{1 + e^{- \boldsymbol{w}^T \boldsymbol{x}_i}} \eqsp,  \quad 1 \leq i \leq I \eqsp,
\end{align*}
where $a$ and $b$ are hyperparameters (shape and inverse scale, respectively) that we assume to be fixed. We thus have $p(y, \data) = p_0(y) \prod_{i=1}^I p(c_i | \boldsymbol{x}_i , y)$ with $p_0(y) = \prod_{l=1}^L p_0(w_l|\beta) p_0(\beta)$ and as the sigmoid does not admit a conjugate exponential prior, $p(\data)$ is intractable in this model.
\end{ex}
One way to bypass this problem is to introduce a variational density $q$ in some tractable density family $\mathcal{Q}$ and to find $q^\star$ such that
$$
q^\star = \arginf_{q \in \mathcal{Q}} \diverg \couple[\PQ][\PP] \eqsp,
$$
where $\PP$ and $\PQ$ denote the probability measures on $(\Yset, \Ysigma)$ with corresponding associated density $p(\cdot|\data)$ and $q$. This optimisation problem still involves the (unknown) normalisation constant $p(\data)$ however it can easily be transformed into the following equivalent optimisation problem
$$
q^\star = \arginf_{q \in \mathcal{Q}} \int_\Yset \falpha \left(\dfrac{q(y)}{p(y, \data)}\right) {p(y, \data)} \nu(\rmd y) \eqsp,
$$
which does not involve the marginal likelihood $p(\data)$ anymore (see for example \cite{2016arXiv160100670B} and \cite{2016arXiv160202311L, NIPS2017_6866}).
The core of Variational Inference methods then consists in designing approximating families $\mathcal{Q}$ which allow efficient optimisation and which are able to capture complicated structure inside the posterior density. Typically, $q$ belongs to a parametric family $q = k_\theta$ where $\theta$ is in a certain parametric space $\Tset$, that is the minimisation occurs over the set of densities
$$
\set{y \mapsto k_\theta(y)}{\theta \in \Tset} \eqsp.
$$
In this paper, we offer to perform instead a minimization over
$$
\set{y \mapsto \int_\Tset \mu(\rmd \theta) k_\theta(y)}{\mu \in \mathsf{M}} \eqsp,
$$
where $\mathsf{M}$ is a convenient subset of $\meas{1}(\Tset)$, the set of probability measures on $\Tset$ (and in this case, we equip $\Tset$ with a $\sigma$-field denoted by $\Tsigma$). In doing so, we extend the minimizing set to a larger space since a parameter $\theta$ can be identified with its associated Dirac measure $\delta_\theta$. Similarly, a mixture model composed of $\lrcb{\theta_1, \ldots, \theta_J} \in \Tset^J$ will correspond to taking $\mu$ as a weighted sum of Dirac measures.

More formally, let us consider a measurable space $(\Tset,\Tsigma)$. Let $p$ be a measurable positive function on $(\Yset,\Ysigma)$ and $K:(\theta,A) \mapsto \int_A k(\theta,y)\nu(\rmd y)$ be a Markov transition
kernel on $\Tset \times \Ysigma$ with kernel density $k$ defined on
$\Tset \times \Yset$. Moreover, for all $\mu\in\meas{1}(\Tset)$, for all
$y \in \Yset$, we denote $\mu k(y) = \int_\Tset \mu(\rmd \theta) k(\theta, y)$ and we define
\begin{align}\label{eq:def-psi}
\Psif(\mu) = \int_\Yset \falpha \left(\dfrac{\mu k(y)}{p(y)}\right) {p(y)} \nu(\rmd y) \eqsp.
\end{align}
Note that $p$, $k$ and $\nu$ appear as well in $\Psif(\mu)$ i.e $\Psif(\mu) = \Psif(\mu;p,q,\nu)$, but we drop them for notational ease and when no ambiguity occurs. Notice also that we replaced $k_\theta(y)$ by $k(\theta, y)$ to comply with usual kernel notation. We consider in what follows the general optimisation problem
\begin{align}
\label{eq:objective}
\arginf_{\mu \in \mathsf{M}} \Psif(\mu)\eqsp,
\end{align}
and in practice, we will choose $p(y) = p(y, \data)$.

At this stage, a first remark is that the convexity of $\Psif$ is straightforward from the convexity of $\falpha$. Therefore, a simple yet powerful consequence of enlarging the variational family is that the optimisation problem now involves the {\em convex} mapping
\begin{equation*}
\mu \mapsto \Psif(\mu)= \int_\Yset \falpha \left(\frac{\mu k(y)}{p(y)}\right) {p(y)} \nu(\rmd y) \eqsp,
\end{equation*}
whereas the initial optimisation problem was associated to the mapping $\theta \mapsto \int_\Yset \falpha \left(\frac{k_\theta (y)}{p(y)}\right) {p(y)} \nu(\rmd y)$, which is not necessarily convex.


We now move on to \Cref{section:exact:aei}, where we describe the \aei descent and state our main theoretical results.

\section{The \aei descent}
\label{section:aei}





\label{section:exact:aei}
\subsection{An iterative algorithm for optimising $\Psif$}

Throughout the paper we will assume the following conditions on $k$, $p$ and $\nu$.

\begin{hyp}{A}
  \item\label{hyp:positive}
  The density kernel $k$ on $\Tset\times\Yset$, the function $p$ on $\Yset$ and
  the $\sigma$-finite measure $\nu$ on $(\Yset,\Ysigma)$ satisfy, for all
  $(\theta,y) \in \Tset \times \Yset$, $k(\theta,y)>0$, $p(y)>0$ and
  $\int_\Yset p(y) \nu(\rmd y)<\infty$.
\end{hyp}

Under \ref{hyp:positive}, we immediately obtain a lower bound on $\Psif$.
\begin{lemma}\label{lemma:bound}
Suppose that \ref{hyp:positive} holds. Then, for all $\mu \in
\meas{1}(\Tset)$, we have
$$
\Psif(\mu) \geq \tilde{\falpha}\left(\int_\Yset p(y) \nu(\rmd y)\right)> -\infty \eqsp,
$$
where $\tilde{\falpha}$ is defined on $(0, \infty)$ by $\tilde{\falpha}(u) = u \falpha(1/u)$.
\end{lemma}
\begin{proof} Since $\tilde{\falpha}(u) = u \falpha(1/u)$, we have
$$
\Psif(\mu) = \int_\Yset \tilde{\falpha}\left(\frac{p(y)}{\mu k(y)}\right) \mu k(y) \nu(\rmd y) \eqsp.
$$
Recalling that $\falpha$ and hence $\tilde{\falpha}$, is convex on $\mathbb{R}_{> 0}$, Jensen's inequality applied to $\tilde{\falpha}$ yields $\Psif(\mu) \geq \tilde{\falpha}\lr{\int_\Yset p(y) \nu(\rmd y)}>-\infty$.
\end{proof}

\begin{rem}
\label{rem:assumption1}
Assumption \ref{hyp:positive} can be extended by discarding the assumption
that $p(y)$ is positive for all $y \in \Yset$. As it complicates the expression
of the constant appearing in the bound without increasing dramatically the
degree of generality of the results, we chose to maintain this assumption for
the sake of simplicity.
\end{rem}

Thus, if there exists a sequence of probability measures $\set{\mu_n}{n \in \nstar}$ on $(\Tset,\Tsigma)$ such that $\Psif(\mu_1)<\infty$ and $\Psif(\mu_n)$ is non-increasing with $n$, \Cref{lemma:bound} guarantees that this sequence converges to a limit in $\rset$. We now focus on constructing such a sequence $\set{\mu_n}{n \in \nstar}$.

For this purpose, let $\mu \in \meas{1}(\Tset)$. We introduce the one-step transition of the \aei descent which can be described as an {\em expectation} step and an {\em iteration} step: 

\begin{algorithm}
\caption{\em Exact \aei descent one-step transition}
\label{algo:aei}
\noindent
\begin{enumerate}
\item \underline{Expectation step} : \setlength{\parindent}{1cm} $\bmuf(\theta)= \mathlarger{\int_\Yset} k(\theta,y)  \falpha' \left(\dfrac{\mu k(y)}{p(y)}\right) \nu(\rmd y)$
\item \underline{Iteration step} : \setlength{\parindent}{1cm} $\iteration (\mu)(\rmd \theta)= \dfrac{\mu(\rmd \theta) \cdot \GammaAlpha(\bmuf(\theta)+ \cte)}{\mu(\GammaAlpha(\bmuf + \cte))}$
\end{enumerate} \
\end{algorithm}
Given a certain $\cte \in \rset$, a certain function $\GammaAlpha$ which takes its values in $\rset_{>0}$ and an initial measure $\mu_1 \in\meas{1}(\Tset)$ such that $\Psif(\mu_1) < \infty$, the iterative sequence of probability measures $(\mu_n)_{n \in \nstar}$ is then defined by setting
\begin{equation}
\label{eq:def:mu}
\mu_{n+1}= \iteration (\mu_n)\;,\qquad n\in\nstar \eqsp.
\end{equation}

A first remark is that under~\ref{hyp:positive} and for all $\alpha \in \rset \setminus \lrcb{1}$, $\bmuf$ is well-defined. As for the case $\alpha = 1$, we will assume in the rest of the paper that
$\bmuf[\mu][1](\theta)$ is finite for all $\mu \in \meas{1}(\Tset)$ and $\theta \in \Tset$.
The iteration $\mu \mapsto \iteration (\mu)$ is thus well-defined if moreover we have
\begin{align}\label{eq:admiss}
\mu(\GammaAlpha(\bmuf + \cte)) < \infty \eqsp.
\end{align}

A second remark is that we recover the Infinite-Dimensional Entropic Mirror Descent algorithm applied to the Kullback-Leibler (and more generally to the $\alpha$-divergence) objective function by choosing $\GammaAlpha$ of the form
$$
\GammaAlpha(v) = e^{-\eta v} \eqsp.
$$ We refer to \cite[Appendix A]{pmlr-v97-hsieh19b} for some theoretical background on the Infinite-Dimensional Entropic Mirror Descent. In this light, $\bmuf$ can be understood as the gradient of $\Psif$. Algorithm \ref{algo:aei} then consists in applying a transform function $\GammaAlpha$ to the gradient $\bmuf$ and projecting back onto the space of probability measures.

In the rest of the section, we investigate some core properties of the aforementioned sequence of probability measures $(\mu_n)_{n\in\nstar}$. We start by establishing conditions on $(\GammaAlpha, \cte)$ such that the \aei descent diminishes $\Psif(\mu_n)$ at each iteration for all $\mu_1 \in \meas{1}(\Tset)$ satisfying $\Psif(\mu_1) < \infty$.

\subsection{Monotonicity}

To establish that the \aei descent diminishes $\Psif(\mu_n)$ at each iteration, we first derive a general lower-bound for the difference $\Psif(\mu)- \Psif(\zeta)$. Here, $(\zeta,\mu)$ is a couple of probability measures where $\zeta$ is dominated by $\mu$ which we denote by $\zeta \preceq \mu$. This first result involves the following useful quantity
\begin{align}\label{eq:Aalpha}
  A_\alpha \eqdef \int_\Yset  \nu(\rmd y)  \int_\Tset \mu(\rmd \theta) k(\theta,y)  \falpha' \left( \frac{g(\theta)\mu k(y) }{p(y)} \right) \left[ 1 - g(\theta) \right]\eqsp,
\end{align}
where $g$ is the density of $\zeta$ w.r.t $\mu$, i.e. $\zeta(\rmd
\theta)=\mu(\rmd \theta) g(\theta)$.

\begin{lemma}\label{lem:fondam}
Assume \ref{hyp:positive}. Then, for all $\mu,\zeta\in\meas{1}(\Tset)$ such that $\zeta \preceq \mu$ and $\Psif(\mu) < \infty$, we have
\begin{equation} \label{eq:bound:fondam}
A_\alpha \leq  \Psif(\mu) - \Psif(\zeta) \eqsp.
\end{equation}
Moreover, equality holds in \eqref{eq:bound:fondam} if and only if $\zeta=\mu$.
\end{lemma}
\begin{proof} To prove \eqref{eq:bound:fondam}, we introduce the intermediate function
$$
h_\alpha(\zeta, \mu) = \int_\Yset \nu(\rmd y) p(y)  \int_\Tset \frac{\mu(\rmd \theta) k(\theta,y)}{\mu k(y)} \falpha \left( \frac{g(\theta)\mu k(y)}{p(y)} \right) \eqsp.
$$
Then, the convexity of $\falpha$ combined with Jensen's inequality implies that
\begin{align}\label{eq:bound1}
h_\alpha(\zeta, \mu) \geq \int_\Yset   \nu(\rmd y) p(y) \falpha \left( \frac{\int_\Tset \mu(\rmd \theta) k(\theta,y) g(\theta)}{p(y)} \right) = \Psif(\zeta) \eqsp.
\end{align}
Next, set $u_{\theta, y} = \frac{g(\theta)\mu k(y)}{p(y)}$ and $v_y = \frac{\mu k(y)}{p(y)}$. Since the function $\falpha$ is convex, we have that for all $\theta\in\Tset$, for all $y \in \Yset$, $\falpha(v_y) \geq \falpha(u_{\theta, y}) + \falpha'(u_{\theta,y})(v_y -u_{\theta,y})$, that is
\begin{align}\label{eq:bound2}
\falpha\left(\frac{\mu k(y)}{p(y)} \right) \geq \falpha\left(\frac{ g(\theta)\mu k(y)}{p(y)} \right) + \falpha'\left(\frac{g(\theta) \mu k(y)}{p(y)} \right)\frac{\mu k(y)}{p(y)}[1 - g(\theta)] \eqsp.
\end{align}
Now integrating over $\Tset$ with respect to $\frac{\mu(\rmd \theta) k(\theta, y)}{\mu k(y)}$ and then integrating over $\Yset$ with respect to $p(y) \nu(\rmd y)$ in \eqref{eq:bound2} yields
\begin{align}\label{eq:bound3}
\Psif(\mu) \geq h_\alpha(\zeta, \mu) + A_\alpha \eqsp.
\end{align}
Combining this result with \eqref{eq:bound1} gives \eqref{eq:bound:fondam}. The case of equality is obtained using the strict convexity of $\falpha$ in \eqref{eq:bound1} and \eqref{eq:bound2} which shows that $g$ is constant
$\mu$-a.e. so that $\zeta=\mu$.
\end{proof}


We now plan on setting $\zeta= \iteration(\mu)$ in \Cref{lem:fondam} and obtain that one iteration of the \aei descent yields $\Psif\circ\iteration(\mu) \leq \Psif(\mu)$. Based on the lower-bound obtained in \Cref{lem:fondam}, a sufficient condition is to prove that taking $g \propto \GammaAlpha(\bmuf + \cte)$ in \eqref{eq:Aalpha} implies $A_\alpha \geq 0$. For this purpose, let us denote by $\Domain$ an interval of $\Rset$ such that for all $\theta \in \Tset$, for all $\mu \in \meas{1}(\Tset)$, $\bmuf(\theta) + \cte$ and $\mu(\bmuf) + \cte \in \Domain$ and let us make an assumption on $(\GammaAlpha, \cte)$.
\begin{hyp}{A}
\item\label{hyp:gamma} The function $\GammaAlpha: \Domain \to\rset_{> 0}$ is decreasing, continuously differentiable and satisfies the inequality
    \begin{align*}
        \left[(\alpha -1)(v-\cte) + 1\right] (\log {\GammaAlpha})'(v) + 1 \geq 0, \quad v \in \Domain \eqsp.
    \end{align*}
\end{hyp}

We now state our first main theorem.

\begin{thm}\label{thm:monotone}
  Assume \ref{hyp:positive} and \ref{hyp:gamma}. Let $\mu\in\meas{1}(\Tset)$ be such that \eqref{eq:admiss} holds and $\Psif(\mu)<\infty$. Then, the two following assertions hold.
\begin{enumerate}[label=(\roman*)]
\item \label{item:mono1} We have  $\Psif \circ \iteration (\mu) \leq \Psif(\mu)$.
\item \label{item:mono2} We have $\Psif \circ \iteration (\mu) =\Psif(\mu)$ if and only if $\mu=\iteration (\mu)$.
\end{enumerate}
\end{thm}

\begin{proof} To prove \ref{item:mono1}, we set $g \propto \GammaAlpha(\bmuf + \cte)$ in \eqref{eq:Aalpha} and we will show that $A_\alpha \geq 0$. Then, the proof is concluded by setting $\zeta=\iteration (\mu)$ in \Cref{lem:fondam} as
\begin{equation}\label{eq:monotone:one}
\Psif \circ \iteration (\mu) \leq \Psif(\mu)- A_\alpha \leq \Psif(\mu) \eqsp.
\end{equation}
We study the cases $\alpha = 1$ and $\alpha \in \rset \setminus \lrcb{1}$ separately.

\begin{enumerate}[label=(\alph*),wide=0pt, labelindent=\parindent]
  \item Case $\alpha = 1$. In this case $\falpha[1]'(u) = \log u$ and we have
\begin{align*}
  A_1 &= \int_\Yset \nu(\rmd y) \int_\Tset \mu(\rmd \theta) k(\theta,y)  \log \left( \frac{g(\theta) \mu k(y)}{p(y)} \right) \left[ 1 - g(\theta) \right]  \\
  &= \int_\Yset  \nu(\rmd y) \int_\Tset \mu(\rmd \theta) k(\theta,y) \left[ \log g(\theta) + \falpha[1]' \left( \frac{ \mu k(y)}{p(y)} \right) \right] \left[ 1 - g(\theta) \right] \\
  &= \int_\Tset \mu(\rmd \theta) \left[\log g(\theta) + \int_\Yset k(\theta,y) \falpha[1]' \left( \frac{ \mu k(y)}{p(y)} \right)  \nu(\rmd y) \right] \left[ 1 - g(\theta) \right] \\
  &= \int_\Tset \mu(\rmd \theta) \left[ \log g(\theta) + \bmuf[\mu][1](\theta) + \cte \right] \left[ 1 - g(\theta) \right] \eqsp.
\end{align*}
where we used that $\mu[\kappa (1-g)] = 0$ in the last equality. Setting $\tgamma (v) = \GammaAlpha(v) / \mu(\GammaAlpha(\bmuf[\mu][1] + \cte))$ for all $v \in \Domain[1]$, we have $g=\tgamma \circ (\bmuf[\mu][1] + \cte)$. Let us thus consider the probability space $(\Tset,\Tsigma,\mu)$ and let $V$ be the random variable $V(\theta)=\bmuf[\mu][1](\theta) + \cte$. Then, $\PE[1-\tgamma(V)] = 0$ and we can write
$$
A_1 = \PE[ (\log \tgamma (V) + V) (1-\tgamma(V))] = \Cov(\log \tgamma (V) + V, 1- \tgamma(V)) \eqsp.
$$ 

Under \ref{hyp:gamma} with $\alpha =1$, $v \mapsto \log \tgamma (v) + v$ and $v \mapsto 1-\tgamma(v)$ are increasing on $\Domain[1]$ which implies $A_1\geq 0$.


  \item Case $\alpha \in \rset \setminus \lrcb{1}$. In this case $\falpha'(u) = \frac{1}{\alpha-1} [ u^{\alpha-1}- 1]$ and we have
  \begin{align*}
A_\alpha &= \int_\Yset \nu(\rmd y) \int_\Tset \mu(\rmd \theta) k(\theta,y) \frac{1}{\alpha-1} \left[ \left( \frac{ g(\theta)\mu k(y)}{p(y)} \right)^{\alpha-1} - 1 \right] \left[ 1 - g(\theta) \right]  \\
&= \int_\Yset \nu(\rmd y) \int_\Tset \mu(\rmd \theta) k(\theta,y) \frac{1}{\alpha-1} \left( \frac{\mu k(y)}{p(y)} \right)^{\alpha-1} g(\theta)^{\alpha -1} \left[ 1 - g(\theta) \right]  \\
&= \int_\Tset \mu(\rmd \theta) \left[\bmuf(\theta) + \frac{1}{\alpha-1} \right] g(\theta)^{\alpha -1} \left[ 1 - g(\theta) \right] \eqsp.
\end{align*}
Again, setting $\tgamma (v) = \GammaAlpha(v) / \mu(\GammaAlpha(\bmuf + \cte))$ for all $v \in \Domain$, we have $g=\tgamma \circ (\bmuf[\mu] + \cte)$. Let us consider the probability space $(\Tset,\Tsigma,\mu)$ and let $V$ be the random variable $V(\theta)=\bmuf(\theta)+ \cte$. Then, we have $\PE[1-\tgamma(V)] = 0$ and setting $\cte' = \cte - \frac{1}{\alpha-1}$ we can write
$$
A_\alpha = \PE[ (V - \cte')\tgamma^{\alpha-1}(V) (1-\tgamma(V))] = \Cov((V - \cte')\tgamma^{\alpha-1}(V), 1- \tgamma(V)) \eqsp.
$$
Under \ref{hyp:gamma} with $\alpha \in \rset \setminus \lrcb{1}$, $v \mapsto (v - \cte')\tgamma^{\alpha-1}(v)$ and $v \mapsto 1- \tgamma(v)$ are increasing on $\Domain$ which implies $A_\alpha\geq 0$.

\end{enumerate}

Let us now show \ref{item:mono2}. The {\em if} part is obvious. As for the {\em only if} part,  $\Psif \circ \iteration (\mu) =\Psif(\mu)$  combined with \eqref{eq:monotone:one} yields
$$
\Psif \circ \iteration (\mu) = \Psif(\mu) - A_\alpha \eqsp,
$$
which is the case of equality in \Cref{lem:fondam}. Therefore, $\iteration (\mu) = \mu$.
\end{proof}

\textbf{Possible choices for $(\GammaAlpha, \cte)$.} At this stage, we have established conditions on $(\GammaAlpha, \cte)$ such that $\Psif \circ \iteration(\mu) \leq \Psif(\mu)$ and identified the case of equality. Notice in particular that the inequality in \ref{hyp:gamma} is free from the parameter $\cte$ when $\alpha = 1$, which implies that the function $\GammaAlpha(v) = e^{-\eta v}$ satisfies \ref{hyp:gamma} for all $\eta \in (0,1]$. As a consequence, the case of the Entropic Mirror Descent with the forward Kullback-Leibler divergence as objective function is included in this framework.

One can also readily check that $\GammaAlpha(v) = [ \left(\alpha -1 \right)v + 1]^{\eta/(1-\alpha)}$ satisfies \ref{hyp:gamma} for all $\alpha \in \Rset \setminus \lrcb{1}$, for all $\cte$ such that $(\alpha-1)\cte \geq 0$ and for all $\eta \in (0,1]$. We will refer to this particular choice of $\GammaAlpha$ as the {\em Power Descent} thereafter. These two examples are summarized in Table \ref{table:admiss} below. \newline

\begin{table}[hbtp]
\caption{Examples of allowed $(\GammaAlpha, \cte)$ in the \aei descent according to \Cref{thm:monotone}.}
\label{table:admiss}
\centering
\setlength\extrarowheight{4.5pt}
\begin{tabular}{|l|l|l|}
\hline
Divergence considered & \multicolumn{2}{|l|}{Possible choices for $(\GammaAlpha, \cte)$}  \\
\hline
{\em Forward KL} ($\alpha = 1$) & $\GammaAlpha(v) = e^{-\eta v}$, $\eta \in (0,1]$  & any $\cte$ \\
\hline
{\em $\alpha$-divergence} with $\alpha \in \rset \setminus \lrcb{1}$ & $\GammaAlpha(v) =  [ \left(\alpha -1 \right)v + 1]^{\frac{\eta}{1-\alpha}}$, $\eta \in(0,1]$ & $(\alpha-1)\cte \geq 0$ \\
\hline
\end{tabular}
\end{table}

\textbf{Improving upon \Cref{lem:fondam}.} In the following Lemma, we derive an explicit lower-bound for $\Psif(\mu) - \Psif\circ \iteration(\mu)$ in terms of the variance of $\bmuf$. Let us thus consider the probability space $(\Tset,\Tsigma,\mu)$ and denote by $\Var_\mu$ the associated variance operator. 

\begin{lemma}\label{lem:mono:refined} Assume \ref{hyp:positive} and \ref{hyp:gamma}. Let $\mu\in\meas{1}(\Tset)$ be such that \eqref{eq:admiss} holds and $\Psif(\mu)<\infty$. Then,
\begin{align}\label{eq:mono:refined}
 \frac{\ctemono}{2} \Var_\mu \left( \bmuf\right) \leq \Psif(\mu) - \Psif\circ \iteration(\mu) \eqsp,
\end{align}
where
$$
\ctemono \eqdef \inf_{v \in \Domain} \lrcb{ \left[(\alpha -1) (v-\cte) + 1\right] (\log \GammaAlpha)'(v) + 1 } \times \inf_{v \in \Domain}  - \GammaAlpha'(v) \eqsp. $$
\end{lemma}

The proof of \Cref{lem:mono:refined} builds on the proof of \Cref{thm:monotone} and can be found in \Cref{subsec:lem:mono:refined}. \newline

\Cref{lem:mono:refined} can be interpreted in the following way: provided that $\ctemono > 0$, \eqref{eq:mono:refined} states that the case of equality is reached if and only if the variance of the gradient $\bmuf$ equals zero. Such a result, which holds for any transform function $\GammaAlpha$ satisfying \ref{hyp:gamma}, quantifies the improvement after one step of the \aei descent.

Interestingly, monotonicity properties akin to \Cref{lem:mono:refined} have previously been derived under stronger smoothness assumptions in the context of Projected Gradient Descent steps. For example, in the particular case where the objective function $f$ is assumed to be $\beta$-smooth on $\rset$, for all $u \in \Rset$ it holds (see for example \cite[Equation 3.5]{MAL-050}) that
$$
\frac{1}{\beta} \| \nabla f(u) \|^2 \leq f(u) - f\left(u- \frac{1}{\beta} \nabla f(u)\right) \eqsp.
$$
This result is then used to obtain improved convergence rates for the Projected Gradient Descent algorithm. Consequently, we are next interested in proving a rate of convergence for the Exact \aei descent by leveraging \Cref{lem:mono:refined}. 



\subsection{Convergence} \label{subsec:convergence} Let $\mu_1 \in \meas{1}(\Tset)$. We want to study the limiting behavior of the Exact \aei descent for the iterative sequence of probability measure $(\mu_n)_{n\in \nstar}$ defined by \eqref{eq:def:mu}. To do so, we first introduce the two following useful quantities
\begin{align*}
\cteinf^{-1} &\eqdef \inf_{v \in \Domain} (- \log \GammaAlpha)' (v) \quad \mbox{ and } \quad \ctesup^{-1} \eqdef \inf_{v \in \Domain} \GammaAlpha(v) \eqsp. 
\end{align*}
We define $\meas{1, \mu_1}(\Tset)$ as the set of probability measures dominated by $\mu_1$. Next, we strengthen the assumptions on $\GammaAlpha$ as follows.
\begin{hyp}{A}
\item\label{hyp:gamma2} The function $\GammaAlpha: \Domain \to\rset_{> 0}$ is $L$-smooth and the function $- \log \GammaAlpha$ is concave increasing.
\end{hyp}
We are now able to derive our second main result.

\begin{thm}
\label{thm:admiss}
Assume \ref{hyp:positive}, \ref{hyp:gamma} and \ref{hyp:gamma2}. Further assume that $\ctemono$, $\cteinf > 0$ and that $0 <\inf_{v \in \Domain} \GammaAlpha(v) \leq \sup_{v \in \Domain} \GammaAlpha(v) < \infty$. Moreover, let $\mu_1\in\meas{1}(\Tset)$ be such that
$\Psif(\mu_1)<\infty$. Then, the following assertions hold.

\begin{enumerate}[label=(\roman*)]
\item \label{item:admiss1} The sequence $(\mu_n)_{n\in\nstar}$ defined by \eqref{eq:def:mu} is well-defined and the sequence
$(\Psif (\mu_n))_{n\in\nstar}$ is non-increasing.

\item \label{item:admiss2} For all $N \in \nstar$, we have
\begin{align}\label{eq:rate}
   \Psif(\mu_N) - \Psif(\mu^\star) \leq \frac{\cteinf}{N} \left[ KL\couple[\mu^\star][\mu_1] + L\frac{ \ctesup}{\ctemono} \Delta_1 \right]\eqsp,
\end{align}
where $\mu^\star$ is such that $\Psif(\mu^\star) = \inf_{\zeta \in \meas{1, \mu_1}(\Tset)} \Psif(\zeta)$ and where we have defined $\Delta_1 = \Psif(\mu_1) - \Psif(\mu^\star)$ and $KL\couple[\mu^\star][\mu_1] = \int_\Tset \log \left(\frac{\rmd\mu^\star}{\rmd\mu_1} \right) \rmd\mu^\star$.
\end{enumerate}
\end{thm}
The proof of \Cref{thm:admiss}, which as hinted previously brings into play \Cref{lem:mono:refined}, is deferred to \Cref{subsec:admiss}. We now wish to comment on the constants appearing in \eqref{eq:rate} and in particular the two constants $KL\couple[\mu^\star][\mu_1]$ and $\Delta_1$ (since the remaining constants $\ctemono$, $\cteinf$, $\ctesup$ and $L$ all involve the function $\GammaAlpha$, which has not been chosen yet in \Cref{thm:admiss}).

To do so, we consider in \Cref{ex:SimplexFramework} the finite-dimensional case where $\mu_1$ is a weighted sum of dirac measures. As we shall explain in more details later on in \Cref{sec:sto}, this case is of particular relevance to us as our procedure can then be used to optimise the mixture weights of any given mixture model. \newpage

\begin{ex}[Simplex Framework] \label{ex:SimplexFramework} Let $J \in \nstar$, let $(\theta_1, \ldots, \theta_J) \in \Tset^J$ and let us consider $\mu_1 = J^{-1} \sum_{j=1}^{J} \delta_{\theta_j}$. Then, $\mu^\star$ is of the form $\sum_{j=1}^{J} \lambda_j^\star \delta_{\theta_j}$ where $(\lambda_1^\star, ..., \lambda_J^\star)$ belongs to the simplex of dimension $J$. Moreover, the two quantities $KL\couple[\mu^\star][\mu_1]$ and $\Delta_1$ can easily be bounded in terms of $J$. Indeed, using that $\log u \leq u - 1$ for all $u > 0$ and that $\sum_{j=1}^{J} \lambda_j^{\star 2} \leq 1$, we obtain that
\begin{align*}
KL\couple[\mu^\star][\mu_1] &= \sum_{j=1}^{J} \lambda_j^\star \log \lambda_j^\star + \log J \\
& \leq \log J \eqsp.
\end{align*}
As for $\Delta_1$, we have by convexity that
\begin{align*}
\Delta_1 & \leq [\mu_1 - \mu^\star](\bmuf[\mu_1]) \eqsp
\end{align*}
and, using Pinsker's inequality as well as the bound on $KL\couple[\mu^\star][\mu_1]$ we have established just above, we can deduce
\begin{align*}
\Delta_1 & \leq [\mu_1 - \mu^\star](\bmuf[\mu_1]- \PE_{\mu_1} \left[ \bmuf[\mu_1] \right]) \\
& \leq \sqrt{2} \sqrt{KL\couple[\mu^\star][\mu_1]}  \max_{1 \leq j, j' \leq J} |\bmuf[\mu_1](\theta_j) - \bmuf[\mu_1](\theta_{j'})| \\
& \leq \sqrt{2 \log J} \max_{1 \leq j, j' \leq J} |\bmuf[\mu_1](\theta_j) - \bmuf[\mu_1](\theta_{j'})| \eqsp.
\end{align*}
\end{ex}
In the next Theorem, we state several practical examples of couples $(\GammaAlpha, \cte)$ which satisfy the assumptions from \Cref{thm:admiss}.

\begin{thm} \label{thm:admiss:spec} Assume \ref{hyp:positive}. Define $\binfty \eqdef \sup_{\theta \in \Tset, \mu \in \meas{1}(\Tset)} |\bmuf(\theta)|$ and assume that $\binfty < \infty$. Let $(\GammaAlpha, \cte)$ belong to any of the following cases.

\begin{enumerate}[label=(\roman*)]
\item \label{item:admiss:kl} Forward Kullback-Leibler divergence ($\alpha = 1)$: $\GammaAlpha(v) = e^{-\eta v}$, $\eta \in (0,1)$ and $\cte$ is any real number (Entropic Mirror Descent);

\item \label{item:admiss:alpha} Reverse Kullback-Leibler ($\alpha = 0$) and $\alpha$-Divergence with $\alpha \in \rset \setminus \lrcb{0,1}$:
\begin{enumerate}[label=(\alph*)]
\item\label{item:admiss-alpha-a} $\GammaAlpha(v) = e^{-\eta v}$, $\eta \in (0,\frac{1}{|\alpha-1| \binfty + 1})$ and $\cte$ is any real number (Entropic Mirror Descent);
\item\label{item:admiss-alpha-b} $\GammaAlpha(v) = [ \left(\alpha -1 \right)v + 1]^{\frac{\eta}{1-\alpha}}$, $\eta \in (0,1]$, $\alpha >1$ and $\cte > 0$ (Power Descent);
\end{enumerate}
\end{enumerate}
Let $\mu_1\in\meas{1}(\Tset)$ be such that
$\Psif(\mu_1)<\infty$. Then, the sequence $(\mu_n)_{n\in\nstar}$ defined by \eqref{eq:def:mu} is well-defined and the sequence $(\Psif (\mu_n))_{n\in\nstar}$ is non-increasing with a convergence rate characterized by \eqref{eq:rate}.
\end{thm}

The proof of \Cref{thm:admiss:spec} can be found in \Cref{subsec:thm:admiss:spec}. In terms of assumptions, we only require the gradients of the function $\Psif$ to be bounded in $l_\infty$-norm, which is a standard assumption, and the objective function to be finite at the starting measure $\mu_1$, i.e $\Psif(\mu_1) < \infty$, which again is a mild assumption that can even be discarded for all $\alpha \neq 0$ (see \Cref{rep:hyp:admiss:spec} of \Cref{sec:addRes}).

Let us now illustrate the benefits of our approach with an example where the different constants appearing in \eqref{eq:rate} are bounded explicitly and where we compare the convergence rate we obtain with typical Mirror Descent convergence results from the optimisation literature.

\begin{ex}[Simplex framework and forward Kullback-Leibler] \label{ex:SimplexfKL} Let $J \in \nstar$, let $(\theta_1, \ldots, \theta_J) \in \Tset^J$ and let us consider $\mu_1 = J^{-1} \sum_{j=1}^{J} \delta_{\theta_j}$. In addition, let $\alpha = 1$ and $\GammaAlpha(v) = e^{-\eta v}$ with $v \in \Domain = [- \binfty[1] + \cte, \binfty[1] + \cte]$ and $\cte \in \Rset$. Then, we have $\ctemono[1] = (1-\eta)\eta e^{-\eta \binfty - \eta \cte}$, $\cteinf[1] = \eta^{-1}$, $\ctesup[1] = e^{\eta \binfty + \eta \cte}$ and $L = \eta^2 e^{\eta \binfty - \eta \cte}$.

In the particular case of the Entropic Mirror Descent, the constant $\cte$ does not appear in the update formula  \eqref{eq:def:mu} due to the normalisation, so we can choose it however we want without impacting the convergence of the algorithm. Notice then that by choosing $\cte = -3 \binfty$ and based on \Cref{ex:SimplexFramework}, we obtain the following convergence rate for all $\eta \in (0,1)$
$$
\Psif(\mu_N) - \Psif(\mu^\star) \leq \frac{\log J}{\eta N} + \frac{\sqrt{2\log J} \binfty}{(1-\eta)N} \eqsp.
$$
\end{ex}
Thus, in the particular case of \Cref{ex:SimplexfKL}, the dominant term in \eqref{eq:rate} with respect to the dimension $J$ of the simplex is in $\log J$ so that we achieve an overall $O(\frac{\log J}{N})$ convergence rate. Furthermore, the range of possible values for $\eta$ is stated explicitly, since the result holds for all $\eta \in (0,1)$. \newline

This is an improvement compared to standard Mirror Descent results, which under similar assumptions only provide an $O(1/\sqrt{N})$ convergence rate and assume an $O(1/\sqrt{N})$ learning rate (see \cite{BECK2003167} or \cite[Theorem 4.2.]{MAL-050}). Indeed, Projected Gradient Descent and Entropic Mirror Descent typically achieve an $O(\sqrt{J/N})$ and $O(\sqrt{\log (J)/N})$ convergence rate respectively in the Simplex framework. This means that \Cref{thm:admiss:spec} improves with respect to both $N$ and $J$ compared to Projected Gradient Descent and that it improves with respect to $N$ for the Entropic Mirror Descent with a small cost in terms of the dimension $J$ of the simplex.

Moreover, while accelerated versions of the Mirror Descent (e.g Mirror Prox, see \cite{Nemirovski04prox-methodwith} or \cite[Theorem 4.4.]{MAL-050}) also yield an $O(1/N)$ convergence rate, they require the objective function to be sufficiently smooth, an additional assumption that we have bypassed when deriving our results. \newline

The case of the Power Descent for $\alpha < 1$ is not included in \Cref{thm:admiss:spec}. This case is trickier and must be handled separately in order to obtain the convergence of the algorithm. For this purpose, we first introduce the following additive set of assumptions
\begin{hyp}{A}
  \item\label{hyp:compact}
\begin{enumerate}[label=(\roman*)]
\item \label{item:theta:one}$\Tset$ is a compact metric space and $\Tsigma$ is
  the associated Borel $\sigma$-field;
\item \label{item:theta:two} for all $y \in \Yset$, $\theta \mapsto k(\theta,y)$ is continuous;
\item \label{item:theta:four} we have $\int_\Yset {\sup_{\theta \in \Tset}k(\theta,y)} \times \sup_{\theta' \in \Tset} \lr{\frac{k(\theta',y)}{p(y)}}^{\alpha-1} \nu(\rmd y)<\infty$.
\end{enumerate}
If $\alpha = 0$, assume in addition that $\int_\Yset {\sup_{\theta \in \Tset} \left| \log\lr{\frac{k(\theta,y)}{p(y)}} \right|} p(y) \nu(\rmd y)<\infty$.
\end{hyp}
%

Here, condition \ref{hyp:compact}-\ref{item:theta:four} implies that $\bmuf(\theta)$ and $\Psif(\mu)$ are uniformly bounded with respect to $\mu$ and $\theta$, which is rather weak condition under \ref{hyp:compact}-\ref{item:theta:one} since we consider a supremum taken over a compact set (and $\Tset$ will always be chosen as such in practice). We then have the following theorem, which states that the possible weak limits of $(\mu_n)_{n\in\nstar}$ correspond to the global infimum of $\Psif$.

\begin{thm}
  \label{thm:repulsive} Assume \ref{hyp:positive} and
  \ref{hyp:compact}. Let $\alpha < 1$, $\cte \leq 0$ and set $\GammaAlpha(v) = [ \left(\alpha -1 \right)v + 1]^{\eta/(1-\alpha)}$ for all $v \in \Domain$. Then, for all
  $\zeta\in\meas{1}(\Tset)$, any $\eta > 0$ satisfies \eqref{eq:admiss} and $\Psif(\zeta) < \infty$.

  Let $\eta \in (0,1]$. Further assume that there exists $\mu_1,\muf \in\meas{1}(\Tset)$ such that the (well-defined) sequence $(\mu_n)_{n\in\nstar}$ defined by \eqref{eq:def:mu} weakly converges to $\muf$ as
  $n\to\infty$. Then the following assertions hold
\begin{enumerate}[label=(\roman*)]
\item \label{item:rep1bis} $(\Psif(\mu_n))_{n\in\nstar}$ is non-increasing,
\item\label{item:rep1} $\muf$ is a fixed point of $\iteration$,
\item\label{item:rep2} $\Psif(\muf)=\inf_{\zeta \in \meas{1,\mu_1}(\Tset)} \Psif(\zeta)$.
\end{enumerate}
\end{thm}

The proof of \Cref{thm:repulsive} is deferred to \Cref{subsec:proof-repulsive}. Intuitively, we expect $\muf$ to be a fixed point of $\iteration$ based on \Cref{thm:monotone}. The core difficulty of the proof is then to prove Assertion \ref{item:rep2} and to do so, we proceed by contradiction: we assume there exists $\mubar \in \meas{1, \mu_1}(\Tset)$ such that $\Psif(\muf) > \Psif(\mubar)$ and we contradict the fact that $(\mu_n)_{n\in\nstar}$ converges to a fixed point. \newline

The impact of \Cref{thm:admiss:spec} and \Cref{thm:repulsive} is twofold: not only our results improve on the $O(1/\sqrt{N})$ convergence rates previously established for Mirror Descent algorithms but they also allow us to go beyond the typical Entropic Mirror Descent framework by introducing the Power Descent.

Another interesting aspect is that the range of allowed values for the learning rate $\eta$ is given explicitly in some cases (namely, the Power Descent and the Entropic Mirror Descent with the forward Kullback-Leibler). This is in contrast with usual Mirror Descent convergence results where the optimal learning rate depends on $\binfty$, the Lipschitz constant of $\Psif$, which might be unknown in practice.

The results we obtained thus far are summarized in Table \ref{table:admiss2} below.

\begin{table}[hbtp]
\caption{Examples of allowed $(\GammaAlpha, \cte)$ in the \aei descent according to \Cref{thm:admiss:spec} and \Cref{thm:repulsive}.}
\label{table:admiss2}
\centering
\setlength\extrarowheight{4.5pt}
\begin{tabular}{|l|l|l|}
\hline
Divergence considered & \multicolumn{2}{|l|}{Possible choice of $(\GammaAlpha, \cte)$}  \\
\hline
{\em Forward KL} ($\alpha = 1$) & $\GammaAlpha(v) = e^{-\eta v}$, $\eta \in (0,1)$  & any $\cte$ \\
\hline
\multirow{2}{4.5cm}{{\em $\alpha$-divergence} with $\alpha \in \rset \setminus \lrcb{1}$ \\ } &
$\GammaAlpha(v) = e^{-\eta v}$, $\eta \in (0, \frac{1}{|\alpha-1| \binfty + 1})$ & any $\cte$ \\
\cline{2-3} & $\alpha>1$, $\GammaAlpha(v) =  [ \left(\alpha -1 \right)v + 1]^{\frac{\eta}{1-\alpha}}$, $\eta \in(0,1]$ & $\cte > 0$ \\
\cline{2-3} & $\alpha<1$, $\GammaAlpha(v) =  [ \left(\alpha -1 \right)v + 1]^{\frac{\eta}{1-\alpha}}$, $\eta \in(0,1]$ & $\cte \leq 0$ \\
\hline
\end{tabular}
\end{table}

As Algorithm \ref{algo:aei} typically involves an intractable integral in the Expectation step, we now turn to a Stochastic version of this algorithm.

\section{Stochastic \aei descent} \label{sec:sto} We start by introducing the notation for the Stochastic version of Algorithm \ref{algo:aei}. Let $M \in \nstar$ and let $\mu \in  \meas{1}(\Tset)$. The Stochastic \aei descent algorithm one-step transition is defined as follows.

 \begin{algorithm}
\caption{\em Stochastic \aei descent one-step transition}
\label{algo:sto}
\noindent
\begin{enumerate}
\item \underline{Sampling step} : \setlength{\parindent}{1cm} Draw independently $Y_1, \ldots, Y_M \sim \mu k$
\item \underline{Expectation step} : \setlength{\parindent}{1cm} $\mathlarger \bmufk(\theta) = \dfrac{1}{M} \sum \limits_{m=1}^{M} \dfrac{ k(\theta,Y_m)} { \mu k(Y_m)} \falpha'\lr{\dfrac{ \mu k(Y_m)}{p(Y_m)}}$
\item \underline{Iteration step} : \setlength{\parindent}{1cm} $\iterationK (\mu) (\rmd \theta) = \dfrac{\mu(\rmd \theta) \cdot \GammaAlpha (\bmufk(\theta) + \cte)}{\mu( \GammaAlpha(\bmufk + \cte))}$
\end{enumerate} \
\end{algorithm}

Let us now denote by $(\Omega, \mathcal F,\mathbb P)$ the underlying probability space and by $\PE$ the associated expectation operator. Given $\hmu_1 \in \meas{1}(\Tset)$, the Stochastic version of the Exact iterative scheme defined by \eqref{eq:def:mu} is then given by
\begin{equation}
\label{eq:def:mu:sto}
\hmu_{n+1}= \iterationK (\hmu_n)\;,\qquad n\in\nstar \eqsp,
\end{equation}
where we have defined for all $\theta \in \Tset$ and for all $n \geq 1$,
\begin{equation} \label{eq:def:bmufk}
\bmufk[\hmu_n](\theta) = \dfrac{1}{M} \sum \limits_{m=1}^{M} \dfrac{ k(\theta,Y_{m, n+1})} { \hmu_n k(Y_{m, n+1})} \falpha'\lr{\dfrac{ \hmu_n k(Y_{m, n+1})}{p(Y_{m, n+1})}}
\end{equation}
with $Y_{1,n+1}, \ldots,  Y_{M,n+1} \overset{\mathrm{i.i.d}}{\sim} \hmu_n k$ conditionally on $\mathcal F_n$ and where $\mathcal F_1 = \emptyset$ and $\mathcal{F}_n = \sigma( Y_{1,2}, \ldots , Y_{M,2}, \ldots,$ $Y_{1,n}, \ldots, Y_{M,n})$ for $k\geq 2$. Notice that we use $\hmu_n k$ as a sampler instead of $k(\theta, \cdot)$ in \eqref{eq:def:bmufk}. As our algorithm optimises over $\mu$, sampling with respect to $\hmu_n k$ is not only cheaper computationally, but it also gives preference to the interesting regions of the parameter space. \newline

A first idea to study this algorithm is to adapt \Cref{thm:admiss} to the Stochastic case. This can be done for the Entropic Mirror Descent and a bound on $\PE[\Psif(\hmu_n) - \Psif(\mu^\star)]$ of the form $O(1/N) + O(1/\sqrt{M})$ can be derived for a wide range of constant learning rates $\eta$ (see \Cref{sec:adaptSto} for the formal statement of the result and its proof). Maintaining an $O(1/N)$ bound however requires $M \geq N^2$, which yields an overall computational cost of order $N^3$. Another option consists in adapting \cite{doi:10.1137/070704277} to our framework. This option involves a learning rate policy $(\eta_n)_{n\in\nset}$ and notably yields an $O(1/\sqrt{N})$ bound for a constant policy $\eta_n = \eta_0/\sqrt{N}$, as written in \Cref{thm:emd:sto} below. \newpage

\begin{thm}\label{thm:emd:sto} Assume \ref{hyp:positive}. Let $M \in \nstar$ and let $\hmu_1 \in \meas{1}(\Tset)$. Given a sequence of positive learning rates $(\eta_n)_{n \in \nset}$, we let $(\hmu_n)_{n\in \nstar}$ be defined by $\frac{\rmd \hmu_{n+1}}{\rmd \hmu_n} \propto e^{-\eta_n\bmufk[\hmu_n]}$ and we set $w_n = \frac{\eta_n}{\sum_{n=1}^{N} \eta_n}$, $n \geq 1$. Further assume that
\begin{align}\label{bound_unif}
\cteLipStoInf \eqdef \lr{\sup_{\mu \in \meas{1}(\Tset)} \int_\Yset \sup_{\theta, \theta ' \in \Tset} \frac{k(\theta,y)^2}{k(\theta ', y)} \left| \falpha ' \lr{ \frac{\mu k(y)}{p(y)}} \right|^2 \nu(\rmd y)}^{1/2} < \infty \eqsp,
\end{align}
and define $\Psif(\mu^\star) = \inf_{\zeta \in \meas{1, \hmu_1}(\Tset)} \Psif(\zeta)$. Then, for any $N \in \nstar$,
\begin{align}\label{eq:rate:StoMD}\PE\lrb{\Psif\lr{\sum_{n=1}^N w_n \hmu_n}-\Psif(\mu^\star)}\leq \frac{\cteLipStoInf^2 \sum_{n=1}^N \eta_n^2 /2}{\sum_{n=1}^{N} \eta_n} + \frac{KL(\mu^\star || \hmu_1)}{\sum_{n=1}^{N} \eta_n} \eqsp,
\end{align}
In particular, the decreasing policy $\eta_n = \eta_0/\sqrt{n}$ yields an $O(\log(N)/\sqrt{N})$ bound in \eqref{eq:rate:StoMD}. Furthermore, the constant policy $\eta_n = \eta_0/\sqrt{N}$ yields an $O(1/\sqrt{N})$ bound in \eqref{eq:rate:StoMD}, which is minimal for $\eta_0 = \cteLipStoInf^{-1} \sqrt{2 KL\couple[{\mu}^\star][{\hmu}_1]}$.
\end{thm}
The proof of \Cref{thm:emd:sto} can be found in \Cref{subsec:CVstoMD} and we give below an example satisfying condition \eqref{bound_unif}.
\begin{ex}\label{ex:GaussianMixtureModels}
Consider the case $\Yset = \rset^d$ and $\alpha = 1$. Let $r> 0$ and let $\Tset = \mathcal{B}(0,r) \subset \rset^d$. Furthermore, let $K_{h}$ be a Gaussian transition kernel with bandwidth $h$ and denote by $k_{h}$ its associated kernel density. Finally, let $p$ be a mixture of two $d$-dimensional Gaussian densities such that $p(y) = 0.5 \frac{e^{-\| y - \theta^\star_1 \|^2 / 2}}{(2 \pi)^{d/2}}+ 0.5 \frac{e^{-\| y - \theta^\star_2 \|^2 / 2}}{(2 \pi)^{d/2}} $ for all $y \in \Yset$ with $\theta_1^\star, \theta_2^\star \in \Tset$. Then, \eqref{bound_unif} holds and we can apply \Cref{thm:emd:sto} (see \Cref{sec:BoundUnif} for details).
\end{ex}
Notice that the $O(1/\sqrt{N})$ convergence rate from \Cref{thm:emd:sto} holds under minimal assumptions on $\Psif$. However, bridging the gap with the $O(1/N)$ convergence rate in \Cref{thm:admiss:spec} typically requires much stronger smoothness and strong-convexity assumptions on $\Psif$ which can be hard to satisfy in practice (see \cite[Theorem 6.2]{MAL-050} for the statement of this result and \cite{chriefabdellatif2019generalization} for an example in Online Variational Inference). Bypassing any of these assumptions like we did in the ideal case in \Cref{thm:admiss:spec} in order to improve on \Cref{thm:emd:sto} constitutes an interesting area of research which is beyond the scope of this paper. \newline 

As for the Stochastic version of Power Descent, we establish the total variation convergence of $\iterationK (\mu)$ towards $\iteration (\mu)$ as $M$ goes to infinity for all $\mu \in \meas{1}(\Tset)$. To do so, consider i.i.d random variables $Y_1, Y_2, \ldots$ with common density $\mu k$ w.r.t $\nu$, defined on the same probability space $(\Omega,\mcf,\PP)$ and denote by $\PE$ the associated expectation operator. We then have \Cref{prop:discretize} below.

\begin{prop}\label{prop:discretize}
Assume \ref{hyp:positive}. Let $\alpha \in \Rset \setminus \lrcb{1}$, $\eta > 0$, $\cte$ be such that $(\alpha-1)\cte \geq 0$ and set $\GammaAlpha(v) = [ \left(\alpha -1 \right)v + 1]^{\eta/(1-\alpha)}$ for all $v \in \Domain$. Let $\mu \in \meas{1}(\Tset)$ be such that $\Psif(\mu)<\infty$, \eqref{eq:admiss} holds and
\begin{equation}
\int_\Tset \mu(\rmd \theta)  \PE\lrb{\lrcb{\frac{k(\theta,Y_1)}{\mu k(Y_1)} \lr{\frac{\mu k(Y_1)}{p(Y_1)}}^{\alpha-1} + (\alpha-1)\cte}^{\frac{\eta}{1-\alpha}}} <\infty\eqsp.
\end{equation}
Then,
$$
\llim_{M \to \infty} \tv{\iterationK (\mu) -\iteration (\mu)}=0\,, \quad \PP-\as
$$
\end{prop}

The proof is deferred to \Cref{subsec:proof:discretize}. The crux of the proof consists in applying a Dominated Convergence Theorem to non-negative real-valued $(\Tsigma \otimes \mcf,{\mathcal B}(\rset_{\geq 0}))$-measurable functions, which requires to consider a Generalized version of the Dominated Convergence Theorem (\Cref{lem:gen:tcd}) and an Integrated Law of Large Numbers (\Cref{lem:lln:integrated}). \newline

\textbf{Mixture Models.} We now address the case where $\hmu_1$ corresponds to a weighted sum of Dirac measures. This case is of particular interest to us since as we shall see, for any kernel $K$ of our choice, the \aei descent procedure simplifies and provides an update formula for the mixture weights of the corresponding mixture model $\hmu_1 K$.

Let $J \in \nstar$ and let $\theta_1, \ldots, \theta_J \in \Tset$ be fixed. We start by introducing the simplex of $\rset^J$
$$
\simplex_J = \set{ \lbd{}= (\lambda_1, \ldots, \lambda_J) \in \rset^J}{ \forall j \in \lrcb{1,\ldots, J}, \eqsp \lambda_j \geq 0 \eqsp \mbox{and} \eqsp \sum_{j=1}^J \lambda_j = 1} \eqsp,
$$
and for all $\lbd{} \in \simplex_J$, we define $\mu_{\lbd{}} \in \meas{1}(\Tset)$ by $\mu_{\lbd{}} = \sum_{j=1}^J \lambda_j \delta_{\theta_j}$.
Then, $\mu_{\lbd{}} k(y) = \sum_{j = 1}^J \lambda_j k(\theta_{j}, y)$ corresponds to a mixture model and if we let $(\hmu_n)_{n\in\nstar}$ be defined by $\hmu_1 = \mu_{\lbd{}}$ and
\begin{equation*}
\hmu_{n+1}= \iterationK (\hmu_n)\;,\qquad n\in\nstar \eqsp,
\end{equation*}
an immediate induction yields that for every $n\in \nstar$, $\hmu_n$ can be expressed as $\hmu_n=\sum_{j=1}^J \lbd[j]{n} \delta_{\theta_j}$ where $\lbd{n}=(\lbd[1]{n},\ldots,\lbd[J]{n}) \in \simplex_J$ satisfies the  initialisation $\lbd{1}=\lbd{}$ and the update formula: for all $n \in \nstar$ and all $j \in \{1,\ldots,J\}$,
\begin{align}\label{eq:iteration:mixture}
\lbd[j]{n+1}= \frac{\lbd[j]{n} \GammaAlpha(\bmufk[\hmu_n](\theta_j) + \cte)}{ \sum_{i=1}^J \lbd[i]{n} \GammaAlpha(\bmufk[\hmu_n](\theta_i) + \cte)} \eqsp,
\end{align}
with $Y_{1, n+1}, \ldots, Y_{M, n+1}$ drawn independently from $\hmu_n k$ conditionally on $\mathcal F_n$ and $\bmufk[\hmu_n](\theta_j)$ is given by \eqref{eq:def:bmufk} for all $j = 1 \ldots J$. This leads to Algorithm \ref{algo:mixture} below. \newline
\SetInd{0.8em}{-1.4em}
\begin{algorithm}[H]
\caption{{\em Mixture Stochastic \aei descent}}
\label{algo:mixture}
\textbf{Input:} $p$: measurable positive function, $K$: Markov transition kernel, $M$: number of samples, $\Theta_J = \{\theta_1, \ldots, \theta_J\} \subset \Tset$: parameter set. \\
\textbf{Output:} Optimised weights $\lbd{}$. \\
\BlankLine
Set $\lbd{} = [\lambda_{1, 1}, \ldots, \lambda_{J,1}]$.\\
\While{not converged}{
\begin{enumerate}[label={}]
\item \underline{Sampling step} : \setlength{\parindent}{1cm} Draw independently $M$ samples $Y_{1}, \ldots, Y_{M}$ from $\mu_{\lbd{}} k$.
\item \underline{Expectation step} : \setlength{\parindent}{1cm} Compute $\boldsymbol{B}_{\lbd{}} = (b_{j})_{1 \leq j \leq J}$ where 
\begin{align}\label{eq:aj}
b_{j} = \dfrac{1}{M} \sum \limits_{m=1}^M \frac{k(\theta_j, Y_{m})}{\mu_{\lbd{}} k(Y_{m})}  \falpha'\left( \frac{\mu_{\lbd{}} k(Y_{m})}{p(Y_{m})} \right)
\end{align} 

\noindent and deduce $\boldsymbol{W}_{\lbd{}}  = (\lambda_j \GammaAlpha(b_{j} + \cte))_{1\leq j\leq J}$ and $w_{\lbd{}}  = \sum_{j=1}^J \lambda_j \GammaAlpha(b_{j} + \cte)$.
\item \underline{Iteration step} : \setlength{\parindent}{1cm} Set
\begin{align*}
\lbd{} \leftarrow \frac{1}{w_{\lbd{}} } \boldsymbol{W}_{\lbd{}}
\end{align*}
\end{enumerate}
}
\end{algorithm}

In this particular framework, most of the computing effort at each step lies within the computation of the vector $(\bmufk[\hmu_n](\theta_j))_{1 \leq j \leq J}$. Interestingly, these computations can also be used to obtain an estimate of the Evidence Lower Bound (resp. the Renyi-Bound \cite{2016arXiv160202311L}) when $p(y) = p(y, \data)$. These two quantities, which are written explicitly in \Cref{rem:RenyiBound} from \Cref{sec:addRes}, allow us to assess the convergence of the algorithm and provide a bound on the log-likelihood (see \cite[Theorem 1]{2016arXiv160202311L}). Note also that if there is a need for very large $J$, one can approximate the summation appearing in $\hmu_{n} k$ using subsampling. \newline

An important point is that Algorithm \ref{algo:mixture} does not require any information on how the $\lrcb{\theta_1, \ldots, \theta_J}$ have been obtained in order to infer the optimal weights as it draws information from samples that are generated from $\mu_{\lbd{}} k$. Since the algorithm leaves $\lrcb{\theta_1, \ldots, \theta_J}$ unchanged throughout the optimisation of the mixture weights (we call it an {\em Exploitation Step}), a natural idea is to combine Algorithm \ref{algo:mixture} with an {\em Exploration step} that modifies the parameter set, which gives Algorithm \ref{algo:adaptive} below. 

\SetInd{0.8em}{-1.4em}
\begin{algorithm}[H]
\caption{{\em Complete Exploitation-Exploration Algorithm}}
\label{algo:adaptive}
\textbf{Input}: $p$: measurable positive function, $\alpha $: $\alpha$-divergence parameter, $(\GammaAlpha, \cte)$: chosen  as per \Cref{table:admiss}, $q_0$: initial sampler, $K$: Markov transition kernel, $(M_t)_t$: number of samples,  $(J_t)_t$: dimension of parameter set. \\
\textbf{Output}: Optimised weights $\lbd{}$ and parameter set $\Theta$. \\
Draw $\thetat[1][0], \ldots, \thetat[J_0][0]$ from $q_0$. Set $t = 0$. \\
\While{not converged}{
\begin{enumerate}[label={}]
\item \underline{Exploitation step} : \setlength{\parindent}{1cm} Set $\Theta = \{ \thetat[1][t], \ldots, \thetat[J_t][t] \}$. Perform Mixture Stochastic \aei descent and obtain $\lbd{}$.
\item \underline{Exploration step} : \setlength{\parindent}{1cm} Perform any exploration step of our choice and obtain $\thetat[1][t+1], \ldots, \thetat[J_{t+1}][t+1]$. Set $t = t+1$.
\end{enumerate}
}
\end{algorithm}
Note that this algorithm is very general, as any Exploration Step can be envisioned. We also have several other levels of generality in our algorithm since we are free to choose the kernel $K$, the $\alpha$-divergence being optimised and we have stated different possible choices for the couple $(\GammaAlpha, \cte)$.

As a side remark, notice also that we recover the mixture weights update rules from the Population Monte Carlo algorithm applied to reverse Kullback-Leibler minimisation \cite{2007arXiv0708.0711D} by considering the Power Descent with $\alpha = 0$ and $\eta = 1$. We have thus embedded this special case into a more general framework.

We now move on to numerical experiments in the next section.

\section{Numerical experiments}
\label{section:applications} In this part, we want to assess how Algorithm \ref{algo:adaptive} performs on both toy and real-world examples. To do so, we first need to specify the kernel $K$ and an algorithm for the Exploration Step.

\paragraph{Kernel} Let $K_{h}$ be a Gaussian transition kernel with bandwidth $h$ and denote by $k_{h}$ its associated kernel density. Given $J \in \nstar$ and $\theta_1, \ldots, \theta_J \in \Tset$, we then work within the approximating family
$$
\set{y \mapsto \mu_{\lbd{}} k_h(y) = \sum_{j= 1}^{J} \lambda_j k_h(y- \theta_j)}{\lbd{} \in \simplex_J }\eqsp.
$$

\paragraph{Exploration Step} At time $t = 1 \ldots T$, we resample among $\lrcb{\thetat[1][t], \ldots, \thetat[J_t][t]}$ according to the optimised mixture weights $\lbd{}$. The obtained sample $\{\thetat[1][t+1], $ $\ldots, \thetat[J_{t+1}][t+1]\}$ is then perturbed stochastically using the Gaussian transition kernel $K_{h_t}$, which gives us our new parameter set. The hyperparameter $h_t$ is adjusted according to the number of particles so that $h_t \propto J_t^{-1/(4 + d)}$, where $d$ is the dimension of the latent space (the optimal rate in nonparametric estimation when the function is at least $2$-times continuously differentiable and the kernel has order $2$ \cite{stone1982}). \newline

Next, we are interested in the choice of $\alpha$. The hyperparameter $\alpha$ allows us to choose between \emph{mass-covering} divergences which tend to cover all the modes ($\alpha \ll 0$) and \emph{mode-seeking} divergences that are attracted to the mode with the largest probability mass ($\alpha \gg 1$), the case $\alpha \in (0,1)$ corresponding to a mix of the two worlds (see for example \cite{divergence-measures-and-message-passing}).

Depending on the learning task, the optimal $\alpha$ may differ and understanding how to select the value of $\alpha$ is still an area of ongoing research. However, the case $\alpha < 1$ presents the advantage that $\bmufk$ is always finite. Indeed, for all $\alpha \in \rset \setminus \lrcb{1}$, we have
$$
\bmuf(\theta) = \frac{1}{\alpha-1} \int_\Yset \frac{k(\theta,y)}{\mu k(y)} \left(\frac{p(y, \data)}{\mu k (y)} \right)^{1- \alpha} \mu k(y) \nu(\rmd y) - \frac{1}{\alpha-1} \eqsp,
$$
and as the dimension grows, the conditions of support are often not met in practice, meaning that there exists $A \in \Ysigma$ such that $p(A, \data) = 0$ and $\mu k (A) > 0$. This implies that whenever $\alpha > 1$ we might have that $\bmufk(\theta) = \infty$ and that the $\alpha$-divergence (or equivalently the Renyi-bound as written in \Cref{rem:RenyiBound} from \Cref{sec:addRes}) is infinite, which is the sort of behavior we would like to avoid. Thus, we restrict ourselves to the case $\alpha \leq 1$ in the following numerical experiments. Note that the limiting case $\alpha = 1$, corresponding to the commonly-used forward Kullback-Leibler objective function, also suffers from this poor behavior, but is still considered in the experiments as a reference.

We now move on to our first example where we investigate the impact of different choices of $\GammaAlpha$. The code for all the subsequent numerical experiments is available at \href{https://github.com/kdaudel/AlphaGammaDescent}{https://github.com/kdaudel/AlphaGammaDescent}.  

\subsection{Toy Example} Following \Cref{ex:GaussianMixtureModels}, the target $p$ is a mixture of two $d$-dimensional Gaussian densities multiplied by a positive constant $Z$ such that
$$
p(y) = Z \times \left[ 0.5 \mathcal{N}(\boldsymbol{y}; -s \boldsymbol{u_d}, \boldsymbol{I_d}) + 0.5 \mathcal{N}(\boldsymbol{y}; s \boldsymbol{u_d}, \boldsymbol{I_d}) \right] \eqsp,
$$
where $\boldsymbol{u_d}$ is the $d$-dimensional vector whose coordinates are all equal to $1$, $s = 2$ and $Z = 2$. $(J_t)_t$ and $(M_t)$ are kept constant equal to $J = M = 100$, $\cte = 0$ and the initial weights are set to be $[1/J, \ldots, 1/J]$. The number of inner iterations in the \aei descent is set to $N=10$ and for all $n =1\ldots N$, we use the adaptive learning rate $\eta_n = \eta_0/\sqrt{n}$ with $\eta_0 = 0.5$. We set the initial sampler to be a centered normal distribution with covariance matrix $5 \boldsymbol{I_d}$, where $\boldsymbol{I_d}$ is the identity matrix. We compare three versions of the \aei algorithm:
\begin{itemize}
\item \underline{$0.5$-Mirror Descent :} $\GammaAlpha(v) = e^{-\eta v}$ with $\alpha = 0.5$,
\item \underline{$0.5$-Power Descent :} $\GammaAlpha(v) = [ \left(\alpha -1 \right)v + 1]^{\eta/(1-\alpha)}$ with $\alpha =0.5$,
\item \underline{$1$-Mirror Descent :} $\GammaAlpha(v) = e^{-\eta v}$ with $\alpha = 1$.
\end{itemize}
For each of them, we run $T= 20$ iterations of Algorithm \ref{algo:adaptive} and we replicate the experiment 100 times for $d = \lrcb{8, 16, 32}$. The results for the 0.5-Mirror and 0.5-Power Descent are displayed on Figure \ref{fig:toy}.

\begin{figure}[!ht]
\caption{Plotted is the average Renyi-Bound for the 0.5-Power and 0.5-Mirror Descent in dimension $d = \lrcb{8, 16,32}$ computed over 100 replicates with $\eta_0 = 0.5$.}
\label{fig:toy}
\begin{center}
\begin{tabular}{ccc}
\includegraphics[scale=0.27]{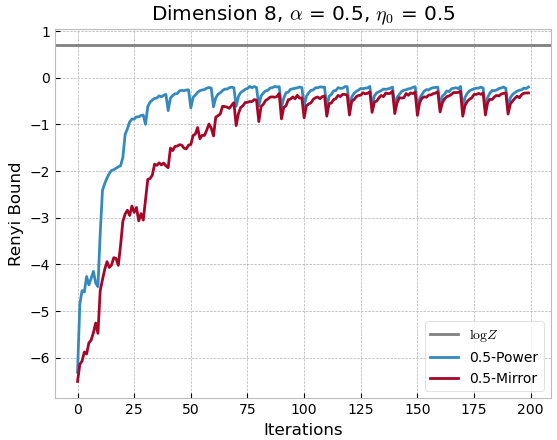}  &
\includegraphics[scale=0.27]{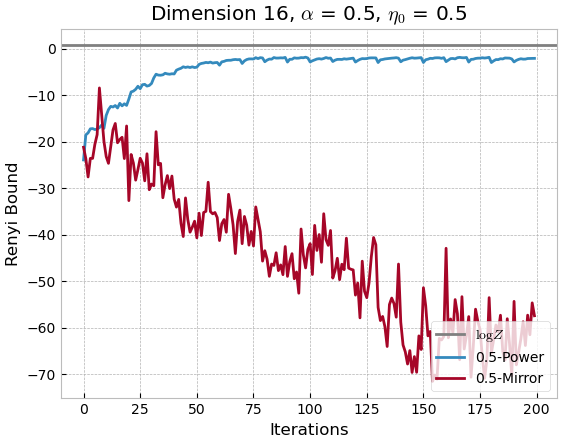} & \includegraphics[scale=0.27]{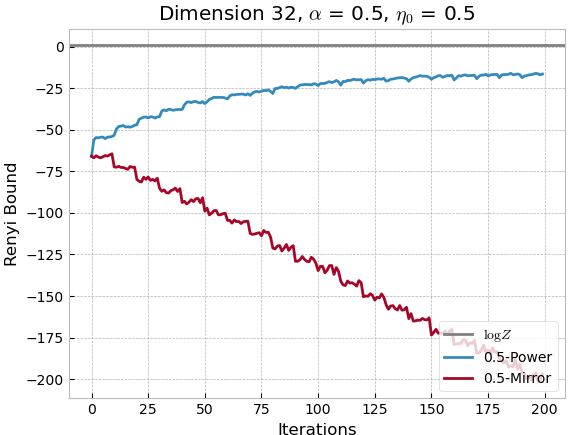}
\end{tabular}
\end{center}
\end{figure}
A first remark is that we are able to observe the monotonicity property from \Cref{thm:admiss} (the Renyi-Bound varies like $\Psif(\mu_n)^{\alpha-1}$) for the 0.5-Power Descent, the jumps in the Renyi-Bound corresponding to an update of the parameter set. Furthermore, we see that the 0.5-Mirror Descent (which would have been the default choice based on the existing optimisation literature) converges more slowly than the 0.5-Power Descent in dimension $8$. An even more striking aspect however is that, as the dimension grows, the 0.5-Mirror Descent is unable to learn and the algorithm diverges.

These two different behaviors for the Power and Mirror Descent can be explained by rewriting the update formulas for any $\alpha <1$ under the form
\begin{align*}
  \mbox{Mirror :} \quad \lbd[j]{n} &\propto e^{\frac{\eta}{1-\alpha} \lrb{ (\alpha-1)\bmuf[\mu_{\lbd{n}}](\theta_j) + (\alpha-1)\cte }} \\
  \mbox{Power :} \quad \lbd[j]{n} &\propto e^{ \frac{\eta}{1-\alpha} \log \lrb{ (\alpha-1)\bmuf[\mu_{\lbd{n}}](\theta_j) + (\alpha-1)\cte}} \eqsp.
\end{align*}
In the Power case, an extra log transformation has been added, which allows to discriminate between small values of $\bmuf[\mu_{\lbd{n}}]$. 
 Since the values of $\bmuf[\mu_{\lbd{n}}]$ tend to get smaller as the dimension grows, the impact of adding an extra log transformation becomes increasingly visible: the Mirror Descent becomes more and more unable to differentiate between the different particles $\lrcb{\theta_1, \ldots, \theta_J}$ and is thus unable to learn. 

Finally, we compare how the 0.5-Power and 1-Mirror Descent perform at approximating the log-likelihood in dimension $d = \lrcb{8,16, 32}$. The results are plotted on Figure \ref{fig:compare}. Again, the 0.5-Power Descent comes across as faster and more stable compared to the 1-Mirror Descent as the dimension grows. Furthermore, it also does not fail in dimension 32, unlike the 1-Mirror Descent.

\begin{figure}[!ht]
\caption{Plotted is the average Log-likelihood for 0.5-Power and 1-Mirror Descent in dimension $d = \lrcb{8, 16,32}$ computed over 100 replicates with $\eta_0 = 0.5$.}
\label{fig:compare}
\begin{center}
\begin{tabular}{ccc}
\includegraphics[scale=0.27]{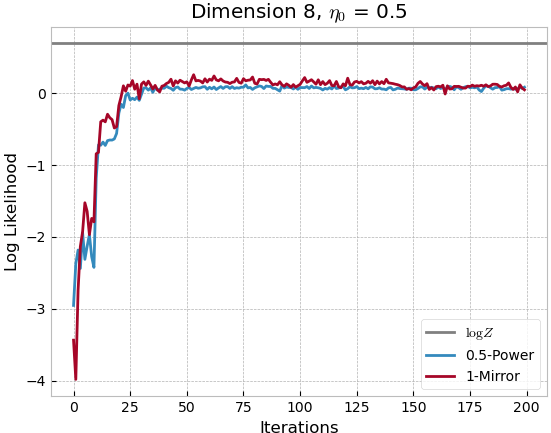} &
\includegraphics[scale=0.27]{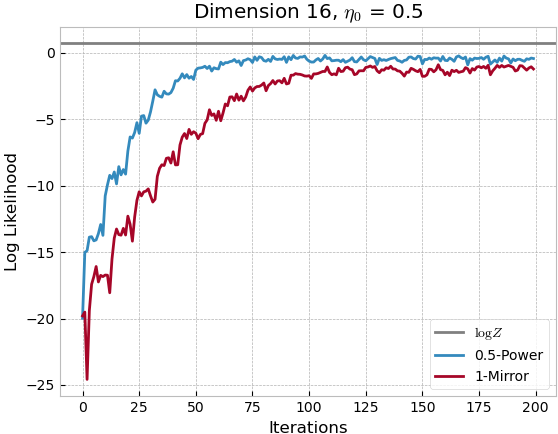} &
\includegraphics[scale=0.27]{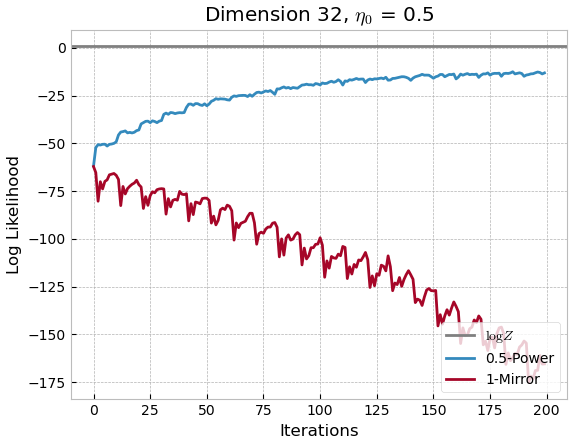}  \\
\end{tabular}
\end{center}
\end{figure}

Consequently, we see on this simple yet illustrative example that the Power Descent is a suitable alternative to the Mirror Descent as the dimension grows.

We are next interested in seeing how the \aei descent performs on a real-data example. Based on the numerical results obtained so far, we rule out the Mirror Descent for $\alpha \leq 1$ and we focus on the Power Descent in our second example.

\subsection{Bayesian Logistic Regression} We consider the Bayesian Logistic Regression from \Cref{ex:BLR} with $a = 1$ and $b = 0.01$. 

We test our algorithm for the \emph{Covertype} dataset ($581,012$ data points and $54$ features, available \href{https://www.csie.ntu.edu.tw/~cjlin/libsvmtools/datasets/binary.html}{here}). 
Computing $p(y, \data)$ constitutes the major computation bottleneck here, since $p(y, \data) = p_0(y) \prod_{i} p(x_i|y)$ with a very large number of data points. We can conveniently address this problem by approximating $p(y, \data)$ with subsampled mini-batches. We adopt this strategy here and consider mini-batches of size $100$.

We set $\alpha = 0.5$, $N = 1$, $T = 500$, $\cte = 0$, $J_0 = M_0 = 20$ and $J_{t+1} = M_{t+1} = J_t + 1$ for $t = 1\ldots T$ in Algorithm \ref{algo:adaptive}. The initial weights in the \aei descent are set to $\lbd{init,t} = [1/J_t, \ldots, 1/J_t]$ and the learning rate is set to $\eta_0 = 0.05$.

One thing that is very specific to the Exploration step that we used to run our experiments (and sampling-based Exploration steps algorithms in general) is that the particles $\lrcb{\thetat[1][t], \ldots, \thetat[J_t][t]}$ are sampled from a known distribution at each Exploration step. This means that we are able to infer information on $\lrcb{\thetat[1][t], \ldots, \thetat[J_t][t]}$ using Importance Sampling (IS) weights. We thus compare the Power \aei descent with a state-of-the-art Adaptive Importance Sampling-based (AIS) algorithm (see for example \cite{oh1992adaptive, kloek1978bayesian, chopin2004, delyon2019safe}).

We initialise $\lrcb{\theta_{1,0}, \ldots, \theta_{J_0,0}}$ by sampling $J_0$ points from the prior $p_0(y) = p_0(\beta)p_0(w|\beta)$ and set $q_0 = p_0$. Given $q_t$ at time $t$, we draw $J_t$ i.i.d samples $\left(\theta_{j,t}\right)_{1\leq j\leq J_t}$ from $q_t$ and we define $q_{t+1}(y) = \sum_{j=1}^{J_t} \lambda_{j, t} k_{h_t}(y-\theta_{j,t})$ where
\begin{equation}
\lambda_{j,t} \propto
\begin{cases}
   \frac{p(\theta_{j,t}, \data)}{q_t(\theta_{j,t})} & \mbox{(AIS)} \eqsp, \\
    \GammaAlpha(\bmufk[\mu_{\lbd{init,t}}](\theta_{j,t}) + \cte) & \mbox{(Power)} \eqsp. \\

\end{cases}
\end{equation}
Note that these two algorithms are computationally equivalent. Indeed, we choose $J_t = M_t$ and $N = 1$, that is we use an average of one sample from each $k(\theta_{j,t}, \cdot)$ to infer information on the relevance of the $\lrcb{\theta_{1,t}, \ldots,  \theta_{J_t,t}}$ with respect to one another. Comparatively, the AIS algorithm uses information directly available by computing the IS weights for $\lrcb{\theta_{1,t}, \ldots,  \theta_{J,t}}$.

We replicate the experiments 100 times. The Accuracy and Log-likelihood averaged over the 100 trials for both algorithms are displayed on \Cref{fig:compareRealData} and we see that the 0.5-Power Descent outperforms the AIS algorithm.

\begin{figure}[!ht]
\caption{Plotted are the average Accuracy and Log-likelihood computed over 100 replicates for Bayesian Logistic Regression on the Covertype dataset for the 0.5-Power Descent and the AIS algorithm.}
\label{fig:compareRealData}
\begin{center}
\begin{tabular}{cc}
\includegraphics[scale=0.35]{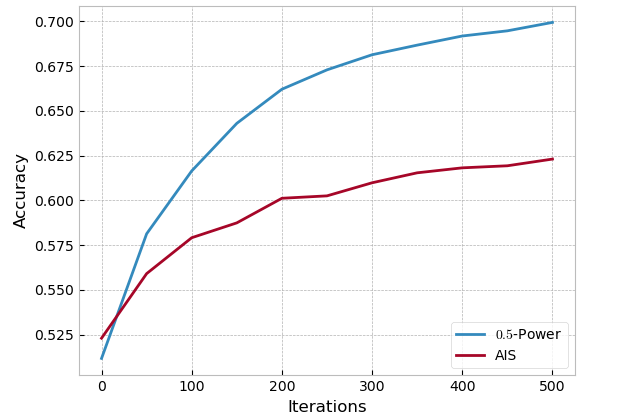} &
\includegraphics[scale=0.35]{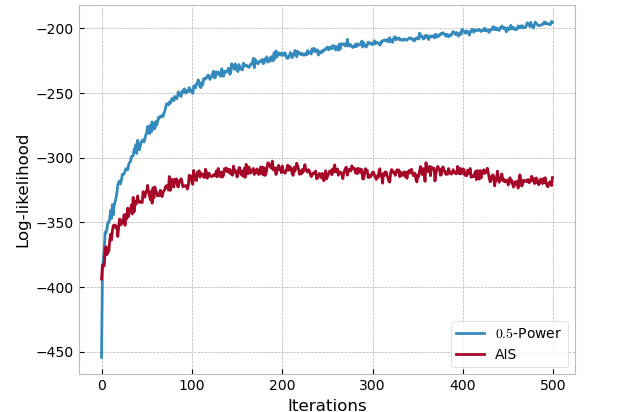}
\end{tabular}
\end{center}
\end{figure}

\section{Conclusion and perspectives} We introduced the \aei descent and studied its convergence. Our framework recovers the Entropic Mirror Descent and allows us to introduce the Power Descent. Furthermore, our procedure provides a gradient-based method to optimise the mixture weights of any given mixture model, without any information on the underlying distribution of the variational parameters. We demonstrated empirically the benefit of going beyond the Entropic Mirror Descent framework by using the Power Descent algorithm instead, which is a more scalable alternative. To conclude, we state several directions to extend our work on both a theoretical and a practical level.

\textit{Convergence rate.} One could seek to establish additional convergence rate results in both the Exact and Stochastic cases, by for example refining the proof of \Cref{thm:admiss} in the Stochastic case.

\textit{Variance Reduction.} One may want to resort to more advanced Monte Carlo methods in the estimation of $\bmuf[\mu_n]$ for variance reduction purposes, such as reusing the past samples in the approximation of $\bmuf[\mu_n]$.

\textit{Exploration Step.} Many other methods could be envisioned as an Exploration step and combined with the \aei descent.

\newpage




\newpage

\appendix

\begin{center}
\textbf{SUPPLEMENTARY MATERIAL}
\end{center}

\section{}

\subsection{Proof of \Cref{lem:mono:refined}}
\label{subsec:lem:mono:refined}

\begin{proof}[Proof of \Cref{lem:mono:refined}] On the probability space $(\Tset,\Tsigma,\mu)$, consider the random variable $U(\theta) = \bmuf(\theta) + \cte$ and let $V$ be an independent copy of $U$. For all $u \in \Domain$, define $\tgamma(u) = \GammaAlpha(u) / \PE[\GammaAlpha]$. Let us now prove that
$$
A_\alpha \geq \frac{\ctemono}{2}\Var_\mu(\bmuf) \eqsp.
$$
We study the cases $\alpha = 1$ and $\alpha \in \rset \setminus \lrcb{1}$ separately.
\begin{enumerate}[label=(\alph*),wide=0pt, labelindent=\parindent]
  \item Case $\alpha = 1$. In this case,
\begin{align*}
A_1 = \Cov(\log \tgamma (U) + U, 1- \tgamma(U)) \eqsp.
\end{align*}
Using that $\PE[1- \tgamma] = 0$, we can rewrite $A_1$ under the form
\begin{align*}
A_1 &= \frac{1}{2} \PE\left[(\log \tgamma(U) + U - \log \tgamma(V) + V)(- \tgamma(U) + \tgamma(V)) \right] \\
&= \frac{1}{2} \PE\left[\frac{\log \tgamma(U) + U - (\log \tgamma(V)+V)}{U - V} \frac{- \tgamma(U) + \tgamma(V)}{U - V} (U-V)^2 \right] \\
&\geq \frac{\ctemono[1]}{2}\Var_\mu(\bmuf[\mu][1]) \eqsp.
\end{align*}

\item Case $\alpha \in \rset \setminus \lrcb{1}$. Set $\cte' = \cte - \frac{1}{\alpha-1}$. In this case,
$$
A_\alpha = \Cov((U - \cte')\tgamma^{\alpha-1}(U), 1- \tgamma(U)) \eqsp,
$$
which, using once again that $\PE[1- \tgamma] = 0$, can be rewritten as
\begin{align*}
A_\alpha &= \frac{1}{2} \PE\left[((U-\cte')\GammaAlpha^{\alpha-1}(U) - (V-\cte')\GammaAlpha^{\alpha-1}(V))(-\GammaAlpha(U) + \GammaAlpha(V)) \right] \\
& = \frac{1}{2} \PE\left[\frac{(U-\cte')\GammaAlpha^{\alpha-1}(U) - (V-\cte')\GammaAlpha^{\alpha-1}(V)}{U - V}\frac{-\GammaAlpha(U) + \GammaAlpha(V)}{U-V} (U-V)^2 \right]\\
& \geq \frac{\ctemono}{2}\Var_\mu(\bmuf) \eqsp.
\end{align*}
Combining with \eqref{eq:bound:fondam} yields \eqref{eq:mono:refined}.
\end{enumerate}
\end{proof}

%

\subsection{Proof of \Cref{thm:admiss}}
\label{subsec:admiss}

\begin{proof}[Proof of \Cref{thm:admiss}] We prove the assertions successively.
\begin{enumerate}[label=(\roman*),wide=0pt, labelindent=\parindent]
 \item The proof of \ref{item:admiss1} simply consists in verifying that we can apply \Cref{thm:monotone}. For all $\mu \in \meas{1}(\Tset)$, \eqref{eq:admiss} holds as we have
     $$
     \mu(\GammaAlpha(\bmuf + \cte)) \leq \mu \left( \sup_{v \in \Domain} \GammaAlpha(v) \right) < \infty,
     $$
      and since at each step $n \in \nstar$, \Cref{thm:monotone} combined with $\Psif(\mu_n) < \infty$ implies that $\Psif(\mu_{n+1}) \leq \Psif(\mu_n) < \infty$, we obtain by induction that $(\Psif(\mu_n))_{n\in\nstar}$ is non-increasing.

 \item For the sake of readability, we only treat the case $\cte = 0$ in the proof of  \ref{item:admiss2}. Note that the case $\cte \neq 0$ unfolds similarly by replacing $\bmuf$ by $\bmuf + \cte$ everywhere in the proof below. Let $n \in \nstar$ and set $\Delta_n = \Psif(\mu_n) - \Psif(\mu^\star)$. We first show that
\begin{align}\label{eq:delta_s}
\Delta_n \leq \cteinf \left[ \int_\Tset \log \left(\frac{\rmd \mu_{n+1}}{\rmd \mu_n} \right)\rmd \mu^\star  + \frac{L}{2} \Var_{\mu_n}(\bmuf[\mu_n]) \ctesup \right] \eqsp.
\end{align}
The convexity of $\falpha$ implies that
\begin{align}\label{eq:rate1}
\Delta_n \leq \int_\Tset \bmuf[\mu_n](\rmd \mu_n - \rmd \mu^\star) = \int_\Tset (\mu_n(\bmuf[\mu_n]) - \bmuf[\mu_n]) \rmd \mu^\star \eqsp.
\end{align}
In addition, the concavity of $- \log \GammaAlpha$ implies that for all $u, v \in \Domain$,
$$
- \log \GammaAlpha(u) \leq - \log \GammaAlpha(v) + (- \log \GammaAlpha)'(v) (u-v) \eqsp,
$$
i.e
$$
(-\log \GammaAlpha)'(v) (v-u) \leq  \log \GammaAlpha(u) - \log \GammaAlpha(v) \eqsp.
$$
Since by assumption $- \log \GammaAlpha$ is increasing, $(- \log \GammaAlpha)'(v) > 0$ and we deduce
\begin{align}\label{eq:rate2}
v-u \leq \frac{\log \GammaAlpha(u) - \log \GammaAlpha(v)}{(-\log \GammaAlpha)'(v)}  \eqsp.
\end{align}
We can apply \eqref{eq:rate2} with $u = \bmuf[\mu_n](\theta)$ and $v = \mu_n(\bmuf[\mu_n])$ which yields
$$
\mu_n(\bmuf[\mu_n])- \bmuf[\mu_n](\theta) \leq \frac{\log \GammaAlpha(\bmuf[\mu_n](\theta)) - \log \GammaAlpha(\mu_n(\bmuf[\mu_n]))}{(- \log \GammaAlpha)'(\mu_n(\bmuf[\mu_n]))}
$$
Now integrating with respect to $\rmd \mu^\star$, we obtain
$$
\Delta_n \leq \frac{1}{(- \log \GammaAlpha)'(\mu_n(\bmuf[\mu_n]))} \int_\Tset \left[ \log \GammaAlpha(\bmuf[\mu_n]) - \log \GammaAlpha(\mu_n(\bmuf[\mu_n])) \right] \rmd \mu^\star \eqsp.
$$
By definition of $\mu^\star$, we have that $\Delta_n \geq 0$ and combining with the fact that $(- \log \GammaAlpha)'(\mu_n(\bmuf[\mu_n])) > 0$, we can deduce
$$
\int_\Tset \left[ \log \GammaAlpha(\bmuf[\mu_n]) - \log \GammaAlpha(\mu_n(\bmuf[\mu_n])) \right] \rmd \mu^\star \geq 0 \eqsp.
$$
Consequently, we obtain
\begin{align}\label{eq:rate3}
\Delta_n &\leq \cteinf \int_\Tset \left[ \log \GammaAlpha(\bmuf[\mu_n]) - \log \GammaAlpha(\mu_n(\bmuf[\mu_n])) \right] \rmd \mu^\star \\
&= \cteinf \int_\Tset \left[ \log \left(\frac{\rmd \mu_{n+1}}{\rmd \mu_n} \right) + \log \mu_n(\GammaAlpha(\bmuf[\mu_n] ))  - \log \GammaAlpha(\mu_n(\bmuf[\mu_n])) \right] \rmd \mu^\star \nonumber \\
&= \cteinf \left[ \int_\Tset \log \left(\frac{\rmd \mu_{n+1}}{\rmd \mu_n} \right)\rmd \mu^\star  + \log \mu_n(\GammaAlpha(\bmuf[\mu_n] ))   - \log \GammaAlpha(\mu_n(\bmuf[\mu_n])) \right] \nonumber \eqsp.
\end{align}
Next, we show that
$$
 \log \mu_n(\GammaAlpha(\bmuf[\mu_n] ))   - \log \GammaAlpha(\mu_n(\bmuf[\mu_n])) \leq \frac{L}{2} \Var_{\mu_n}(\bmuf[\mu_n]) \ctesup\eqsp.
$$
By assumption $\GammaAlpha$ is $L$-smooth on $\Domain$, thus for all $\theta \in \Tset$ and for all $n \in \nstar$,
\begin{multline*}
\GammaAlpha(\bmuf[\mu_n](\theta) ) \leq \GammaAlpha(\mu_n(\bmuf[\mu_n])) + \GammaAlpha'(\mu_n(\bmuf[\mu_n]) )(\bmuf[\mu_n](\theta) - \mu_n(\bmuf[\mu_n])) \\ + \frac{L}{2} \left( \bmuf[\mu_n](\theta) - \mu_n(\bmuf[\mu_n]) \right)^2
\end{multline*}
which in turn implies
$$
\mu_n(\GammaAlpha(\bmuf[\mu_n])) \leq \GammaAlpha(\mu_n(\bmuf[\mu_n])) + \frac{L}{2} \Var_{\mu_n} \left( \bmuf[\mu_n]\right) \eqsp.
$$
Finally, we obtain
$$
\log \mu_n(\GammaAlpha(\bmuf[\mu_n]) ) \leq \log \GammaAlpha(\mu_n(\bmuf[\mu_n]) ) + \log \left(1 +  \frac{L}{2} \frac{\Var_{\mu_n}( \bmuf[\mu_n]) }{\GammaAlpha(\mu_n(\bmuf[\mu_n]) )} \right) \eqsp.
$$
Using that $\log(1+u) \leq u$ when $u \geq 0$ and that $1/\GammaAlpha$ is increasing, we deduce
$$
\log \mu_n(\GammaAlpha(\bmuf[\mu_n] )) \leq \log \GammaAlpha(\mu_n(\bmuf[\mu_n])) + \frac{L}{2} \Var_{\mu_n} \left( \bmuf[\mu_n]\right) \ctesup\eqsp.
$$
which combined with \eqref{eq:rate3} implies \eqref{eq:delta_s}. To conclude, we apply \Cref{lem:mono:refined} to $g = \frac{\rmd \mu_{n+1}}{\rmd \mu_n}$ and combining with \eqref{eq:delta_s}, we obtain
$$
\Delta_n \leq \cteinf \left[ \int_\Tset \log \left(\frac{\rmd \mu_{n+1}}{\rmd \mu_n} \right)\rmd \mu^\star  + \frac{L \ctesup }{\ctemono} \left( \Delta_n - \Delta_{n+1} \right) \right] \eqsp,
$$
where by assumption $\ctemono$, $\cteinf$ and $\ctesup > 0$. As the r.h.s involves two telescopic sums, we deduce
\begin{align}\label{eq:end}
\frac{1}{N} \sum_{n=1}^{N} \Psif(\mu_n) - \Psif(\mu^\star) \leq \frac{\cteinf}{N} \bigg[ KL\couple[\mu^\star][\mu_1] - KL\couple[\mu^\star][\mu_{N + 1}] \left. + L\frac{ \ctesup}{\ctemono} (\Delta_1 - \Delta_{N+1}) \right]
\end{align}
and we recover \eqref{eq:rate} using \ref{item:admiss1}, that $KL\couple[\mu^\star][\mu_{N + 1}] \geq 0$ and that $\Delta_{N+1} \geq 0$.
\end{enumerate}
\end{proof}

\begin{rem}
Note that the convexity of the mapping $\mu \mapsto \Psif(\mu)$ in \eqref{eq:end} implies an $O(1/N)$ convergence rate for $\bar{\mu}_N = \frac{1}{N} \sum_{n=1}^{N} \mu_n$ as well:
\begin{align*}
 \Psif\lr{\bar{\mu}_N} - \Psif(\mu^\star) \leq \frac{\cteinf}{N} \left[ KL\couple[\mu^\star][\mu_1] + L\frac{ \ctesup}{\ctemono} \Delta_1 \right] \eqsp.
\end{align*}
\end{rem}

\subsection{Proof of \Cref{thm:admiss:spec}}
\label{subsec:thm:admiss:spec}

We start with a side note on $\Domain$. A typical choice for $\Domain$ is
\begin{equation}\label{eq:def:Dalpha1}
\Domain = [-\binfty + \cte, \binfty + \cte] \eqsp.
\end{equation}
However, when $\alpha \in \rset \setminus \lrcb{1}$, we might consider instead
\begin{equation}\label{eq:def:Dalpha}
\Domain = \begin{cases}
             [\frac{1}{1-\alpha} + \cte, \binfty+ \cte], & \mbox{if } \alpha > 1 \\
             [-\binfty + \cte, \frac{1}{1-\alpha}+ \cte], & \mbox{if } \alpha < 1
           \end{cases}
\end{equation}
to underline the fact that for all $v \in \Domain$, $(\alpha-1)v + 1 \geq (\alpha-1)\cte$. Unless specified otherwise, we let $\Domain$ be as in \eqref{eq:def:Dalpha} whenever $\alpha \in \rset \setminus \lrcb{1}$.

\begin{proof}[Proof of \Cref{thm:admiss:spec}] Let us recall the different conditions that must be met in order to verify that we can apply \Cref{thm:admiss} in each of the cases mentioned in \Cref{thm:admiss:spec}:

\begin{enumerate}[wide=0pt]
  \item \label{cond:well-def} $0 <\inf_{v \in \Domain} \GammaAlpha(v)$ and $\sup_{v \in \Domain} \GammaAlpha(v) < \infty$.
  \item \label{cond:mono} The function $\GammaAlpha: \Domain \to\rset_{> 0}$ is decreasing, continuously differentiable and satisfies the inequality
    \begin{align*}
        \left[(\alpha -1)(v-\cte) + 1\right] (\log {\GammaAlpha})'(v) + 1 \geq 0 \eqsp.
    \end{align*}
    \item \label{cond:smooth} We have $\ctemono = \inf_{v \in \Domain} \lrcb{ \left[(\alpha -1) (v-\cte) + 1\right] (\log \GammaAlpha)'(v) + 1 } \times $

         $\inf_{v \in \Domain}  - \GammaAlpha'(v) > 0$.
  \item \label{cond:logG} The function $\GammaAlpha: \Domain \to\rset_{> 0}$ is $L$-smooth and the function $- \log \GammaAlpha$ is concave increasing.
  \item \label{cond:pos} $\cteinf = \left(\inf_{v \in \Domain} (- \log \GammaAlpha)' (v)\right)^{-1} > 0$.
\end{enumerate}

\begin{enumerate}[label=(\roman*)]
\item Forward Kullback-Leibler divergence ($\alpha = 1)$: $\GammaAlpha(v) = e^{-\eta v}$, $\eta \in (0,1)$, any real $\cte$. Since the update formula does not depend on $\cte$, there is no constraint on $\cte$ and we assume that $\cte = 0$ for simplicity.

    - Condition \ref{cond:well-def} is satisfied since $\binfty[1]$ is finite. 

    - Condition \ref{cond:mono} is satisfied with $\GammaAlpha'(v) = -\eta e^{-\eta}$ and $(\log \GammaAlpha)'(v) = -\eta$.

    - Condition \ref{cond:smooth} is satisfied with $\ctemono[1] \geq (1-\eta)\eta e^{-\eta \binfty[1]}$.

    - Condition \ref{cond:logG} is satisfied.

    - Condition \ref{cond:pos} is satisfied with $\cteinf[1] = \frac{1}{\eta}$.

\item Reverse Kullback-Leibler ($\alpha = 0$) and $\alpha$-Divergence with $\alpha \in \rset \lrcb{0,1}$:
\begin{enumerate}[label=(\alph*),wide=0pt, labelindent=\parindent]
\item $\GammaAlpha(v) = e^{-\eta v}$, $\eta \in (0, \frac{1}{|\alpha-1| \binfty + 1})$, any real $\cte$. The only difference with the previous case lies in the inequality (i.e Condition \ref{cond:mono}), which can be rewritten for all $v \in \Domain$ as
    $$
    1 \geq \eta \left[(\alpha-1)(v-\cte) +1 \right] \eqsp,
    $$
    Since $0 \leq (\alpha-1)(v-\cte) +1 \leq |\alpha-1| \binfty + 1$, this inequality is then satisfied for $\eta \in (0, \frac{1}{|\alpha-1| \binfty + 1})$.

\item \underline{Case $\alpha > 1$.} $\GammaAlpha(v) = ((\alpha-1)v + 1)^{\frac{\eta}{1-\alpha}}$, $\eta \in (0,1]$ and $\cte$ satisfies $(\alpha-1)\cte > 0$. Then, the condition $(\alpha-1)\cte > 0$ ensures that $\GammaAlpha$ is well-defined on $\Domain$. 
    From there, we deduce:

    - Condition \ref{cond:well-def} is satisfied since $\binfty$ is finite.

    - Condition \ref{cond:mono} is satisfied: $\GammaAlpha'(v) = -\eta ((\alpha-1)v + 1)^{\frac{\eta}{1-\alpha} -1}$, $(\log \GammaAlpha)'(v) = \frac{- \eta}{(\alpha-1) v + 1} $ and the inequality can be rewritten for all $v \in \Domain$ as
    $$
    1 \geq \eta \left[1 - \frac{(\alpha-1)\cte}{(\alpha-1)v +1}\right] \eqsp,
    $$
    which is satisfied for $\eta \in (0, 1]$.

    - Condition \ref{cond:smooth} is satisfied (the condition $(\alpha -1)\cte > 0$ is of crucial importance here).

    - Condition \ref{cond:logG} is satisfied with $(- \log \GammaAlpha)''(v) = \frac{\eta(1-\alpha)}{((\alpha-1)v + 1)^2}$ (note that we need $\alpha > 1$ here).

    - Condition \ref{cond:pos} is satisfied and here again we use that $(\alpha -1)\cte > 0$.
\end{enumerate}
\end{enumerate}
\end{proof}

\subsection{Proof of \Cref{thm:repulsive}}
\label{subsec:proof-repulsive}

In this part, recall that we focus on the particular case $\alpha < 1$, $\cte \leq 0$ and $\GammaAlpha(v) = [ \left(\alpha -1 \right)v + 1]^{\eta/(1-\alpha)}$ for all $v \in \Domain$. In the following, we use the notation
$\mu_n\dlim\muf$ for the weak convergence of measures in $\meas{1}(\Tset)$. For all $\zeta \in \meas{1}(\Tset)$, for all $\theta \in \Tset$, define
$$
g_\zeta(\theta) = (\alpha-1)(\bmuf[\zeta](\theta)+ \cte) + 1 \eqsp.$$
We first derive four useful lemmas.
\begin{lemma} \label{lem:repulsiveOne} Assume \ref{hyp:positive} and \ref{hyp:compact}. Suppose that $\mu_n\dlim\muf$. Then the following assertions hold.
\begin{enumerate}[label=(\roman*)]
\item \label{item:lem1} For all $y \in \Yset$, $\mu_nk(y)$ tends to $\muf k(y)$
    as $n\to\infty$.
\item \label{item:lem2} For all $\zeta \in \meas{1}(\Tset)$, the function
  $\theta \mapsto g_\zeta(\theta)$ is continuous. Furthermore for all
  $\theta \in \Tset$, $g_{\mu_n}(\theta)$ tends to $g_{\muf}(\theta)$ as $n\to\infty$.
\item \label{item:lem2:prim} There exist $0<m_-<m_+<\infty$ such that, for all $\zeta \in \meas{1}(\Tset)$ and $\theta\in\Tset$, $g_\zeta(\theta) \in[m_-,m_+]$.
\item \label{item:lem3} For all continuous, positive and bounded function $h$,
$$\lim_{n \to \infty} \int_\Tset \mu_n(\rmd \theta) \GammaAlpha (\bmuf[\mu_n](\theta)+ \cte) h(\theta) = \int_\Tset \muf(\rmd \theta) \GammaAlpha (\bmuf[\muf](\theta)+ \cte) h(\theta) \eqsp.$$
\end{enumerate}
\end{lemma}

\begin{proof}
We prove the assertions successively.

  \noindent\textbf{Proof of~\ref{item:lem1}.}
For all $y \in \Yset$, the function $\theta \mapsto k(\theta,y)$ is
continuous on a compact set, hence bounded. The weak convergence
 $\mu_n\dlim\muf$ thus implies the pointwise convergence of
  $\mu_n k$ to $\muf k$.

  \noindent\textbf{Proof of~\ref{item:lem2}.}
 For all $\theta \in \Tset$ and  $\zeta \in \meas{1}(\Tset)$, we write
  $$
  g_{\zeta}(\theta)=\int_\Yset a_\zeta(\theta,y) \nu(\rmd y)\;,
  $$
  where we set for all $(\theta, y) \in \Tset \times \Yset$,
  $a_\zeta(\theta,y)=k(\theta,y) \lr{\frac{\zeta k(y)}{p(y)}}^{\alpha-1}$. The continuity of $g_{\zeta}(\theta)$ follows from
  the Dominated Convergence Theorem, since for all $y\in\Yset$, the function
  $\theta \mapsto a_\zeta(\theta,y)$ is continuous on $\Tset$ by \ref{hyp:compact}-\ref{item:theta:two} and for all
  $(\theta, y) \in \Tset \times \Yset$, we have
\begin{align}
\label{eq:repulsiveMinus1}
|a_\zeta(\theta,y)| \leq \sup_{\theta' \in \Tset}k(\theta',y) \times  \sup_{\theta'' \in \Tset} \lr{\frac{k(\theta'',y)}{p(y)}}^{\alpha-1}\;,
\end{align}
which is integrable w.r.t $\nu(\rmd y)$ by \ref{hyp:compact}-\ref{item:theta:four}.
The second part of \ref{item:lem2} is obtained similarly. Using \ref{item:lem1}
and that $u \mapsto u^{\alpha-1}$ is $C^1$, we get that, for all $(\theta, y) \in \Tset \times \Yset$,
\begin{align*}
\llim_{n \to \infty} k(\theta,y) \lr{\frac{\mu_n k(y)}{p(y)}}^{\alpha-1}
  =k(\theta,y)\lr{\frac{ \muf k(y)}{p(y)}}^{\alpha-1} \;,
\end{align*}
i.e $\llim_{n \to \infty} a_{\mu_n}(\theta,y) = a_{\muf}(\theta,y)$. The bound
\eqref{eq:repulsiveMinus1} and
\ref{hyp:compact}-\ref{item:theta:four} provide a domination criterion and we
get that $g_{\mu_n}(\theta)$ tends to $g_{\muf}(\theta)$ as
$n\to\infty$, which concludes the proof of \ref{item:lem2}.

  \noindent\textbf{Proof of~\ref{item:lem2:prim}.}
  For all $(\theta,\zeta) \in \Tset \times \meas{1}(\Tset)$, we have
  $g_{\zeta}(\theta) \in [m_-,m_+]$ where
\begin{align}
\label{eq:m-+}
m_-&\eqdef\int_\Yset \inf_{\theta' \in \Tset}k(\theta',y) \times
     \inf_{\theta'' \in \Tset}
     \lr{\frac{k(\theta'',y)}{p(y)}}^{\alpha -1} \nu(\rmd y) \;, \\
m_+& \eqdef \int_\Yset \sup_{\theta' \in \Tset}k(\theta',y) \times
     \sup_{\theta'' \in \Tset}
     \lr{\frac{k(\theta'',y)}{p(y)}}^{\alpha-1} \nu(\rmd y)  \;.\nonumber
\end{align}
We have that $m_+$ is finite by \ref{hyp:compact}-\ref{item:theta:four}. Furthermore, $u \mapsto u^{\alpha-1}$ does not vanish on
$(0,\infty)$. Together with \ref{hyp:positive}, we thus have that for any $y \in \Yset$, the functions $\theta \mapsto k(\theta,y)$ and $\theta \mapsto \lr{k(\theta,y)/p(y)}^{\alpha-1}$ are continuous and positive on the compact set $\Tset$, from which we deduce that
$m_->0$.

  \noindent\textbf{Proof of~\ref{item:lem3}.}
Using \ref{item:lem2}, the function $\theta \mapsto \GammaAlpha(\bmuf[\muf](\theta)+ \cte)
h(\theta)$ is continuous, and, since $\Tset$ is compact, $\mu_n\dlim\muf$ gives that
\begin{align}
\label{eq:repulsiveZero}
\lim_{n \to \infty} \int_\Tset \mu_n(\rmd \theta) \GammaAlpha(\bmuf[\muf](\theta)+ \cte) h(\theta) = \int_\Tset \muf(\rmd \theta) \GammaAlpha(\bmuf[\muf](\theta) + \cte) h(\theta) \eqsp.
\end{align}
Next we show that
\begin{align}
\label{eq:repulsiveOne}
\lim_{n \to \infty} \int_\Tset \mu_n(\rmd \theta) \left| \GammaAlpha(\bmuf[\mu_n](\theta)+ \cte) - \GammaAlpha(\bmuf[\muf](\theta)+ \cte) \right|  h(\theta) = 0
\end{align}
ie
\begin{align*}
\lim_{n \to \infty} \int_\Tset \mu_n(\rmd \theta) \left| g_{\mu_n}(\theta)^{\frac{\eta}{1-\alpha}} - g_{\muf}(\theta)^{\frac{\eta}{1-\alpha}} \right|  h(\theta) = 0
\end{align*}
Using~\ref{item:lem2:prim}, since $u \mapsto u^{\frac{\eta}{1-\alpha}}$ is Lipschitz on
$[m_-,m_+]$, there exists a constant $C$ such that
\begin{align*}
\mu_n \left[ \left|  g_{\mu_n}(\theta)^{\frac{\eta}{1-\alpha}} - g_{\muf}(\theta)^{\frac{\eta}{1-\alpha}} \right|  h \right] &\leq C \sup_{\theta \in \Tset} h(\theta) \int_\Tset \mu_n(\rmd \theta) \left| g_{\mu_n}(\theta) -g_{\muf}(\theta) \right|\\
 &=C \sup_{\theta \in \Tset} h(\theta)  \int_{\Yset} |a_n(y)| \nu(\rmd y)
\end{align*}
where $a_n(y) \eqdef \mu_n k(y) \lrcb{ \lr{\frac{\mu_n k(y)}{p(y)}}^{\alpha-1} - \lr{\frac{\muf k(y)}{p(y)}}^{\alpha-1}}$. Now, for all $y \in \Yset$,
\begin{align*}
|a_n(y)| &\leq 2 \ {\sup_{\theta \in \Tset}k(\theta,y)}\times \sup_{\theta' \in \Tset} \lr{\frac{k(\theta',y)}{p(y)}}^{\alpha-1} \eqsp,
\end{align*}
which is integrable w.r.t $\nu$ by
\ref{hyp:compact}-\ref{item:theta:four}. Moreover, by \ref{item:lem1}
and by continuity of $u \mapsto u^{\alpha-1}$, we have $\lim_{n \to \infty} a_n(y) = 0$, and
\eqref{eq:repulsiveOne} follows by dominated convergence. Finally, combining
\eqref{eq:repulsiveZero}, \eqref{eq:repulsiveOne} and
\begin{align*}
\mu_n \left[ \GammaAlpha(\bmuf[\mu_n](\theta)+ \cte) h \right] =& \ \mu_n \left[ \GammaAlpha(\bmuf[\mu_n](\theta)+ \cte) h - \GammaAlpha(\bmuf[\muf](\theta)+ \cte)  h \right] \\ &+ \mu_n \left[\GammaAlpha(\bmuf[\muf](\theta)+ \cte)  h \right]
\eqsp,\end{align*}
we obtain
\ref{item:lem3}, and the proof is concluded.
\end{proof}

\begin{lemma}
\label{lem:fixed:repulsive:prelim} Assume \ref{hyp:positive}. Let $\muf, \mu \in \meas{1}(\Tset)$ and assume that there exists $\mubar \in \meas{1, \mu}(\Tset)$ such that $\Psif(\mubar) < \Psif(\muf) $. Then, there exists $\delta > 1$ such that
\begin{align}\label{eq:b_mes}
\mubar(g_{\muf} > \delta \muf(g_{\muf}))>0 \eqsp.
\end{align}
\end{lemma}

\begin{proof}
Let $\zeta, \zeta' \in \meas{1}(\Tset)$. Then, by convexity of $\falpha$ we have,
\begin{align*}
\int_\Tset [\zeta-\zeta'](\rmd \theta) \bmuf[\zeta'](\theta) \leq \Psif(\zeta) - \Psif(\zeta')\eqsp.
\end{align*}
that is
\begin{align}\label{eq:b_ineq}
\int_\Tset [\zeta-\zeta'](\rmd \theta) g_{\zeta'}(\theta) \geq (\alpha-1) \left( \Psif(\zeta) - \Psif(\zeta') \right)\eqsp.
\end{align}

Furthermore, for all $\delta > 1$,
$(\delta -1) \muf(g_{\muf}) \geq 0$.
Let us define $A_\delta = \{ g_{\muf} > \delta \muf(g_{\muf}) \}$ and show that
$\mubar(A_\delta)>0$ for some $\delta>1$. To do so, we proceed by
contradiction. Suppose that $\mubar(A_\delta)=0$ for all $\delta>1$, so that
$$
\mubar[ g_{\muf} - \muf(g_{\muf}) ] = \mubar[ \lr{g_{\muf} - \muf(g_{\muf})}
\indi{A_\delta^c} ] \leq (\delta -1) \muf(g_{\muf}) \;.
$$
Using \eqref{eq:b_ineq}, we get that, for all $\delta>1$,
$$
0< (\alpha-1) \left( \Psif(\mubar) - \Psif(\muf) \right) \leq \mubar[ \lr{g_{\muf} - \muf(g_{\muf})} ]
\leq (\delta -1) \muf(g_{\muf}) \;.
$$
Letting $\delta\downarrow 1$, we obtain a contradiction, which finishes the
proof.
\end{proof}

\begin{lemma}\label{lem:fixed:repulsive} Assume \ref{hyp:positive}. Let $\muf \in \meas{1}(\Tset)$ be a fixed point of
  $\iteration$ and let $\eta > 0$. Let $\mu \in \meas{1}(\Tset)$ and assume that there exists
  $\mubar \in \meas{1, \mu}(\Tset)$ such that
  $\Psif(\muf) > \Psif(\mubar)$.
Then, there exists $\delta > 1$ such that
$$
\mubar\lrcb{\GammaAlpha(\bmuf[\muf]+ \cte) > \delta \muf(\GammaAlpha(\bmuf[\muf] + \cte))}  > 0 \eqsp.
$$
\end{lemma}

\begin{proof} Note that \eqref{eq:admiss} holds for any $\eta > 0$ and $\zeta$ (in
  particular $\zeta=\muf$) by \Cref{lem:repulsiveOne}-\ref{item:lem2:prim}. As
  $\muf$ is a fixed point of $\iteration$, $\gmuf[\muf]$ is $\muf$-almost all
  constant. Consequently, $\muf(g_\muf)^{{\eta}/{1-\alpha}} = \muf(g_\muf^{{\eta}/{1-\alpha}}) = \muf(\GammaAlpha(\bmuf[\muf]+ \cte))$. For all $\delta > 1$, $\delta' \eqdef \delta^{(1-\alpha)/\eta} > 1$ and
\begin{align*}
\mubar\lrcb{\GammaAlpha(\bmuf[\muf]+ \cte) > \delta \muf(\GammaAlpha(\bmuf[\muf] + \cte))} &= \mubar\lrcb{g_\muf > \delta^{(1-\alpha)/\eta} [\muf(g_\muf^{\eta/(1-\alpha)})]^{(1-\alpha)/\eta}} \\
&= \mubar(g_\muf > \delta' \muf(g_\muf) ) \;.
\end{align*}
We conclude by applying \Cref{lem:fixed:repulsive:prelim}.
\end{proof}

\begin{lemma}
  \label{lem:find-B} Assume \ref{hyp:positive} and
  \ref{hyp:compact}.  Let $\eta > 0$, let $\mu_1 \in
  \meas{1}(\Tset)$ and define
  the sequence $(\mu_n)_{n\in\nstar}$ according to \eqref{eq:def:mu}. Suppose
  that $\mu_n\dlim\muf$ for some fixed point  $\muf \in \meas{1}(\Tset)$  of
  $\iteration$.  Further assume there exists
  $\mubar \in \meas{1, \mu_1}(\Tset)$ such that
  $\Psif(\muf) > \Psif(\mubar)$. Then, there exist $\delta >1$ and
  $n \in \nset^*$ such that
$$
\mubar \lr{\bigcap_{m \geq n} \lrcb{\GammaAlpha(\bmuf[\mu_m] + \cte)> \delta \mu_m(\GammaAlpha(\bmuf[\mu_m] + \cte))} }  > 0 \eqsp.
$$
\end{lemma}
\begin{proof}
    First note that the sequence $(\mu_n)_{n\in\nstar}$ is well-defined for any
    $\eta > 0$ by \Cref{lem:repulsiveOne}-\ref{item:lem2:prim}, which implies $\mu_n(\GammaAlpha(\bmuf[\mu_n]+ \cte))>0$ for all $n \in \nstar$. For all $\zeta \in \meas{1}(\Tset)$, set $h_\zeta(\theta) = \GammaAlpha(\bmuf[\zeta](\theta) + \cte)$. We further have that
  \begin{align*}
    \lim_{n\to\infty}\mubar \lr{\bigcap_{m \geq n} \lrcb{h_{\mu_m}>
    \delta \mu_m(h_{\mu_m})} }
    &= \mubar \lr{\bigcup_{n\geq1}\bigcap_{m \geq n} \lrcb{h_{\mu_m}>
      \delta \mu_m(h_{\mu_m})} } \\
&=  \mubar \lr{\set{\theta\in\Tset}{
  \liminf_{n\to\infty}\frac{h_{\mu_n}(\theta)}{\mu_n(h_{\mu_n})}>\delta }}
  \;.
  \end{align*}
  Furthermore, applying \ref{item:lem2} and \ref{item:lem3} in
  \Cref{lem:repulsiveOne}, we have, for all $\theta \in \Tset$,
  $\lim_{n \to \infty} h_{\mu_n}(\theta)=h_{\muf}(\theta)$ and
  $\lim_{n\to\infty}\mu_n(h_{\mu_n})=\muf(h_{\muf})$. Hence, for all
  $\theta \in \Tset$,
$$
\lliminf_{n\to \infty} \frac{h_{\mu_n}(\theta)}{\mu_n(h_{\mu_n})} = \frac{h_{\muf}(\theta)}{\muf(h_{\muf})}\eqsp.
$$
The proof is concluded by applying \Cref{lem:fixed:repulsive}.
\end{proof}

\begin{proof}[Proof of \Cref{thm:repulsive}] Assume \ref{hyp:positive} and \ref{hyp:compact}.

\Cref{lem:repulsiveOne}-\ref{item:lem2:prim} is exactly the first result we want to obtain, that is: for all $\zeta \in \meas{1}(\Tset)$, any $\eta > 0$ satisfies \eqref{eq:admiss} for $\zeta$. Furthermore, $|\Psif(\zeta)| < \infty$ by \ref{hyp:compact}-\ref{item:theta:four}.

Assume that $(\mu_n)_{n\in\nstar}$ weakly converges to $\muf \in
\meas{1}(\Tset)$. First note that \Cref{lem:repulsiveOne}-\ref{item:lem2:prim}
implies that for any $\eta > 0$ the sequence $(\mu_n)_{n\in\nstar}$ is
well-defined and  $\muf$ satisfies \eqref{eq:admiss}. Using \Cref{thm:monotone}, we obtain that the sequence $(\mu_n)_{n\in\nstar}$ is decreasing for all $\eta \in (0,1]$, which gives Assertion \ref{item:rep1bis}.

We now prove Assertions~\ref{item:rep1} and~\ref{item:rep2} successively.

\noindent\textbf{Proof of~\ref{item:rep1}.}
 For all $\zeta \in \meas{1}(\Tset)$ and all $y \in \Yset$, set $a_\zeta(y) =
 \falpha\left(\frac{\zeta k(y)}{p(y)}\right) p(y)$, leading to
\begin{equation} \label{eq:psi:a}
\Psif(\zeta)=\int_{\Yset} a_\zeta(y) \nu(\rmd y) \eqsp.
\end{equation}
Then, for all $y \in \Yset$,
\begin{align}
\label{eq:repulsivite:psi:int}
|a_{\zeta}(y)| \leq  \sup_{\theta \in \Tset} \left| \falpha \left(\frac{k(\theta, y)}{p(y)}\right) \right|  p(y)\eqsp,
\end{align}
which is integrable w.r.t $\nu(\rmd y)$ by \ref{hyp:compact}-\ref{item:theta:four}. Furthermore, recall that for all $y \in \Yset$,
$$
[\iteration(\mu_n)k](y) = \frac{\int_\Tset \mu_n(\rmd \theta) \GammaAlpha(\bmuf[\mu_n](\theta) + \cte) k(\theta, y)}{\int_\Tset \mu_n(\rmd \theta) \GammaAlpha(\bmuf[\mu_n](\theta)+ \cte)}\eqsp.
$$
By applying twice \Cref{lem:repulsiveOne}-\ref{item:lem3} with $h(\theta)=1$ and $h(\theta)=k(\theta,y)$, we have that for all $y \in \Yset$,
\begin{align}
\label{eq:repulsivite:iteration}
\llim_{n \to \infty} [\iteration(\mu_n)k](y) = [\iteration(\muf)k](y) \eqsp.
\end{align}

Now, since $\falpha$ is $C^1$, we obtain from \Cref{lem:repulsiveOne}-\ref{item:lem1} and \eqref{eq:repulsivite:iteration} respectively that for all $y \in \Yset$,
$\lim_{n \to \infty} a_{\mu_n}(y) = a_{\muf}(y)$ and $\lim_{n \to \infty} a_{\iteration(\mu_n)}(y) = a_{\iteration(\muf)}(y)$. Combining with \eqref{eq:repulsivite:psi:int} and \eqref{eq:psi:a} we can thus apply the Dominated Convergence Theorem to obtain
\begin{align}
\label{eq:repulsivite:psi}
\llim_{n \to \infty} \Psif(\mu_n) = \Psif(\muf)
\end{align}
and
\begin{align}
\label{eq:repulsivite:psi:iter}
\llim_{n \to \infty} \Psif(\mu_{n+1}) = \llim_{n \to \infty} \Psif(\iteration(\mu_{n})) = \Psif(\iteration(\muf)) \eqsp.
\end{align}

Finally, \eqref{eq:repulsivite:psi} and \eqref{eq:repulsivite:psi:iter} together yield $\Psif (\muf)  = \Psif \circ \iteration(\muf)$, which in turn implies that $\muf$ is a fixed point of $\iteration$ according to \Cref{thm:monotone}-(ii). 

\noindent\textbf{Proof of~\ref{item:rep2}.}
We prove~\ref{item:rep2} by contradiction. Suppose that $\mu_n\dlim\muf$, where
$\muf$  is a fixed point of $\iteration$ that satisfies
$$
\Psif(\muf)>\inf_{\zeta \in \meas{1,\mu_1}(\Tset)} \Psif(\zeta)\eqsp.
$$
Then, there exists $\mubar \in \meas{1, \mu_1}(\Tset)$ such that $\Psif(\muf) > \Psif(\mubar)$. Now for all $n\in\nstar$, set
$$
B_n = \set{\theta \in \Tset}{\bigcap_{m \geq n} \lrcb{h_{\mu_m}(\theta)> \delta \mu_m(h_{\mu_m})} }\eqsp,
$$
where for all $\zeta \in \meas{1}(\Tset)$, for all $\theta \in \Tset$,
$h_{\zeta}(\theta) \eqdef \GammaAlpha(\bmuf[\zeta](\theta) + \cte)$.  There exists, according to
\Cref{lem:find-B}, for a well chosen $\delta>1$ and a sufficiently large $n_0$
such that $\mubar(B_{n_0})>0$.

Furthermore $\mubar \approx \mu_1$ by definition, where
$\zeta \thickapprox \mu_1$ if and only if for all $A \in \Tsigma$: $\zeta(A)>0$
is equivalent to $\mu_1(A)>0$.  Since $0 < \GammaAlpha(\bmuf[\mu_1](\theta)+ \cte) < \infty$
for $\mu_1$-almost all $\theta \in \Tset$ and
$\frac{\rmd \mu_2}{\rmd \mu_1}\propto \GammaAlpha(\bmuf[\mu_1]+ \cte)$, we also have
$\mu_2 \thickapprox \mu_1$. Then by induction, $\mu_n \thickapprox \mu_1$ for all
$n\in \nstar$. Finally, $\mu_{n_0}(B_{n_0}) > 0$. Moreover, for all
$\theta \in B_{n_0}$ and all $m > n_0$,
$\frac{h_{\mu_m}(\theta)}{\mu_m(h_{\mu_m})}> \delta$ and consequently
\begin{align*}
\mu_m(B_{n_0})=\int_{B_{n_0}} \mu_{m-1}(\rmd \theta)\frac{h_{\mu_{m-1}}(\theta)}{\mu_{m-1}(h_{\mu_{m-1}})}
\geq \delta \mu_{m-1}(B_{n_0})\;.
\end{align*}
By induction on $m$ we get that, for all $m\geq n$,
$\mu_m(B_{n_0})\geq\delta^{m-n_0}\mu_{n_0}(B_{n_0})$. This contradicts the
previously obtain facts that $\delta>1$ and $\mu_{n_0}(B_{n_0})>0$. Therefore
we get a contradiction and the proof is concluded.
\end{proof}

\section{}

\subsection{Adapting \Cref{thm:admiss:spec} in the Stochastic case} \label{sec:adaptSto}

Here, we want to adapt \Cref{thm:admiss:spec} to the Stochastic case for the Entropic Mirror Descent. Given $\hmu_1\in\meas{1}(\Tset)$ with $\Psif(\hmu_1)<\infty$ and letting $(\hmu_n)_{n\in \nstar}$ be defined by \eqref{eq:def:mu:sto}, we will need the following additionnal assumption on the sequence of iterates $(\mu_n)_{n\in \nstar}$, which controls the difference $|\bmufk[\hmu_n](\theta) - \bmuf[\hmu_n](\theta)|$ uniformly with respect to $\theta$.
\begin{hyp}{A}
\item \label{hyp:SupThetaMean} Assume that there exists $\sigma > 0$ such that for all $n \in \nstar$,
$$
\PE \left[ \sup_{\theta \in \Tset}\left|\bmufk[\hmu_n](\theta) - \bmuf[\hmu_n](\theta) \right| \right] \leq \frac{\sigma}{\sqrt{M}} \eqsp.
$$
\end{hyp}
Before stating the result, let us first comment on the validity of this assumption.

\paragraph{Validity of Assumption \ref{hyp:SupThetaMean}, an example} Set $g_n(\theta, y) = \frac{k(\theta, y)}{\hmu_n k(y)} \falpha'\left( \frac{\hmu_n k(y)}{p(y)}\right)$ for all $\theta \in \Tset$, all $n \in \nstar$ and all $y \in \Yset$. In the particular case of the Simplex framework (see \Cref{ex:SimplexFramework}), which is the case we use in practice, \ref{hyp:SupThetaMean} holds with $\sigma = 2 C\sqrt{\frac{\log(2 J)}{2}}$, where $C$ is a positive constant satisfying $|g_n(\theta_j, Y_{m,n+1})| \leq C$ almost-surely for all $j = 1 \ldots J$, all $n \in \nstar$ and all $m = 1 \ldots M$.
\begin{proof} For all $u > 0$, we have by Jensen's inequality that
\begin{align}\label{eq:maxPE}
e^{u \PE \left[ \max_{1 \leq j \leq J} M \left|\bmufk[\hmu_n](\theta_j)  - \bmuf[\hmu_n](\theta_j) \right| \right]} \leq \PE \left[ e^{u \max_{1 \leq j \leq J} M \left|\bmufk[\hmu_n](\theta_j) - \bmuf[\hmu_n](\theta_j) \right|} \right]
\end{align}
Furthermore, Hoeffding's lemma implies
$$
\PE \left[ e^{u \lrcb{ g_n(\theta_j, Y_{m,n+1}) - \bmuf[\hmu_n](\theta_j)} } \right] \leq e^{\frac{u^2 C^2}{2}}
$$ 
and consequently
\begin{align*}
\PE &\left[ e^{u M \lrcb{ \bmufk[\hmu_n](\theta_j) - \bmuf[\hmu_n](\theta_j)} } \right] \leq  e^{\frac{M u^2 C^2}{2}} \eqsp.
\end{align*}
Similarly, we have
\begin{align*}
\PE &\left[ e^{- u M \lrcb{ \bmufk[\hmu_n](\theta_j) - \bmuf[\hmu_n](\theta_j)} } \right] \leq e^{ \frac{M u^2 C^2}{2}} \eqsp.
\end{align*}
which implies
\begin{align*}
\PE \left[e^{u  M \left|\bmufk[\hmu_n](\theta_j)  - \bmuf[\hmu_n](\theta_j) \right| }\right] \leq 2 e^{ \frac{M u^2 C^2}{2} } \eqsp.
\end{align*}
Then, combining with \eqref{eq:maxPE}, we have
\begin{align*}
e^{u  \PE \left[ \max_{1 \leq j \leq J} M \left|\bmufk[\hmu_n](\theta_j)  - \bmuf[\hmu_n](\theta_j) \right| \right]} \leq 2 J e^{\frac{M u^2 C^2}{2}}
\end{align*}
and we obtain
\begin{align*}
\PE \left[ \max_{1 \leq j \leq J} M \left|\bmufk[\hmu_n](\theta_j)  - \bmuf[\hmu_n](\theta_j) \right| \right] \leq \frac{\log(2J)}{u} + \frac{M u C^2}{2} \eqsp.
\end{align*}
Setting $u = \sqrt{\frac{2 \log (2J)}{M C^2}}$ yields the desired result, as we have
\begin{align*}
\PE \left[ \max_{1 \leq j \leq J} \left|\bmufk[\hmu_n](\theta_j)  - \bmuf[\hmu_n](\theta_j) \right| \right] \leq 2 C \sqrt{\frac{\log(2J)}{2 M}}
\end{align*}
\end{proof}
We now state in the next Theorem an $O(1/\sqrt{N} + O(1/\sqrt{M}))$ bound on $\PE[\Psif(\hmu_n) - \Psif(\mu^\star)]$ in the particular case of the Stochastic Entropic Mirror Descent.

\begin{thm} \label{thm:admiss:spec:sto} Assume \ref{hyp:positive}. Let $\hmu_1\in\meas{1}(\Tset)$ be such that $\Psif(\hmu_1)<\infty$, let $(\hmu_n)_{n\in \nstar}$ be defined by \eqref{eq:def:mu:sto} and assume that \ref{hyp:SupThetaMean} holds. Further assume that $\hbinfty \eqdef \sup_{n \in \nstar, \theta \in \Tset} |\bmufk[\hmu_n](\theta)| < \infty$ and let $\GammaAlpha(v) = e^{-\eta v}$. Finally, let $(\eta, \cte)$ belong to any of the following cases.
\begin{enumerate}[label=(\roman*)]
\item Forward Kullback-Leibler divergence ($\alpha = 1)$: $\eta \in (0,1)$ and $\cte$ is any real number;

\item  Reverse Kullback-Leibler ($\alpha = 0$) and $\alpha$-Divergence with $\alpha \in \rset \setminus \lrcb{0,1}$:
 $\eta \in (0,\frac{1}{|\alpha-1| \hbinfty + 1})$ and $\cte$ is any real number;
\end{enumerate}
Then, the sequence $(\hmu_n)_{n\in\nstar}$ is well-defined and for all $N \in \nstar$, we have
\begin{align*}
 \PE\left[\Psif\lr{\frac{1}{N} \sum_{n=1}^{N} \hmu_n} - \Psif(\mu^\star) \right] \leq  \frac{1}{N \eta} \lrb{  KL\couple[\mu^\star][\mu_1] + L\frac{ \ctesup}{\ctemono} \Delta_1} + \frac{1}{\sqrt{M}} \frac{L \ctesup L_{\alpha,4} \sigma}{\eta \ctemono}  \eqsp,
\end{align*}
where $\mu^\star$ is such that $\Psif(\mu^\star) = \inf_{\zeta \in \meas{1, \hmu_1}(\Tset)} \Psif(\zeta)$ and where we have defined $\Delta_1 = \Psif(\hmu_1) - \Psif(\mu^\star)$, $KL\couple[\mu^\star][\hmu_1] = \int_\Tset \log \left(\frac{\rmd\mu^\star}{\rmd\hmu_1} \right) \rmd\mu^\star$ and $L_{\alpha,4} \eqdef \sup_{v \in \Domain} \GammaAlpha(v)^{\alpha-1} \sup_{v\in \Domain} \GammaAlpha(v)^{1-\alpha} \left[ 1 + \frac{\sup_{v\in \Domain} \GammaAlpha(v)}{\inf_{v\in \Domain} \GammaAlpha(v)}\right]$.
\end{thm}

The first step to prove this result is to see what becomes of \Cref{lem:mono:refined} in the Stochastic framework, which we investigate in \Cref{lem:mono:refined:sto} below.
\begin{lemma}\label{lem:mono:refined:sto} Assume \ref{hyp:positive} and \ref{hyp:gamma}. Let $\hmu_1\in\meas{1}(\Tset)$ be such that $\Psif(\hmu_1)<\infty$, let $(\hmu_n)_{n\in \nstar}$ be defined by \eqref{eq:def:mu:sto} and assume that \ref{hyp:SupThetaMean} holds. Further assume that
$L_{\alpha, 4} < \infty$. Then, for all $n \in \nstar$,
\begin{align}\label{eq:mono:refined:sto}
 \frac{\ctemono}{2} \PE \left[ \Var_{\hmu_n} \left( \bmufk[\hmu_n] \right) \right] \leq \PE \left[ \Psif(\hmu_n) - \Psif(\hmu_{n+1}) \right] + L_{\alpha,4} \frac{\sigma}{\sqrt{M}} \eqsp.
\end{align}
\end{lemma}
\begin{proof}We consider the case $\cte = 0$ for simplicity. Set $\hat{g}_n(\theta) = \tgamma(\bmufk[\hmu_n](\theta))$ for all $\theta \in \Tset$ and for all $n \in \nstar$, where $\tgamma(u) = \GammaAlpha(u) / \PE_{\hmu_n}[\GammaAlpha]$. Based on the proof of \Cref{thm:monotone}, we have
\begin{align}\label{eq:interSto}
A_\alpha \leq \Psif(\hmu_n) - \Psif(\hmu_{n+1}) \eqsp,
\end{align}
where
$$
A_\alpha = \begin{cases}
             \int_\Tset \hmu_n(\rmd \theta) \left[ \log \hat{g}_n(\theta) + \bmuf[\hmu_n][1](\theta) + \cte \right] \left[ 1 - \hat{g}_n(\theta) \right] & \mbox{if } \alpha = 1 \\
             \int_\Tset \hmu_n(\rmd \theta) \left[\bmuf[\hmu_n](\theta) + \frac{1}{\alpha-1} \right] \hat{g}_n(\theta)^{\alpha -1} \left[ 1 - \hat{g}_n(\theta) \right], & \mbox{otherwise}.
           \end{cases}
$$
Now defining
$$
E_\alpha = \hmu_n \left(\left[ \bmufk[\hmu_n] - \bmuf[\hmu_n] \right] \hat{g}_n^{\alpha-1} \left[ 1 - \hat{g}_n \right] \right) \eqsp.
$$
and based on the proof of \Cref{lem:mono:refined} we can rewrite \eqref{eq:interSto} as
$$
\frac{\ctemono}{2} \Var_{\hmu_n} \left( \bmufk[\hmu_n] \right)  \leq \Psif(\hmu_n) - \Psif(\hmu_{n+1}) + E_\alpha \eqsp
$$
and we deduce
\begin{align*}
      \PE \left[E_\alpha \right] &= \PE \left[ \hmu_n \left(\left[ \bmufk[\hmu_n] - \bmuf[\hmu_n] \right] \hat{g}_n^{\alpha -1} \left[ 1 - \hat{g}_n \right] \right) \right] \\
      & \leq L_{\alpha, 4} \PE\left[ \hmu_n \left( \PE \left[\left| \bmufk[\hmu_n] - \bmuf[\hmu_n] \right| | \mathcal F_n \right]  \right) \right] \\
      & \leq L_{\alpha, 4}\cdot\frac{\sigma}{\sqrt{M}} \eqsp.
\end{align*}
\end{proof}

Next, we derive the Stochastic version of \Cref{thm:admiss} in the particular case of the Entropic Mirror Descent.

\begin{thm}\label{thm:stoThm2} Assume \ref{hyp:positive}. Set $\GammaAlpha(v) = e^{-\eta v}$ and let $\eta$ be such that \ref{hyp:gamma} and \ref{hyp:gamma2} hold. Let $\hmu_1\in\meas{1}(\Tset)$ be such that $\Psif(\hmu_1)<\infty$, let $(\hmu_n)_{n\in \nstar}$ be defined by \eqref{eq:def:mu:sto} and assume that \ref{hyp:SupThetaMean} holds. Further assume that $\ctemono$, $\cteinf > 0$ and that $0 <\inf_{v \in \Domain} \GammaAlpha(v) \leq \sup_{v \in \Domain} \GammaAlpha(v) < \infty$. Then, for all $N \in \nstar$, we have
\begin{align*}
\PE\left[\Psif\lr{\frac{1}{N} \sum_{n=1}^{N}  \hmu_n} - \Psif(\mu^\star) \right] \leq \frac{1}{N \eta} \lrb{ KL\couple[\mu^\star][\mu_1] + L\frac{ \ctesup}{\ctemono} \Delta_1 } + \frac{1}{\sqrt{M}} \frac{L \ctesup L_{\alpha,4} \sigma}{\eta \ctemono} \eqsp.
\end{align*}
\end{thm}

\begin{proof} We consider the case $\cte = 0$ for simplicity. Let $n \in \nstar$ and set $\Delta_n = \Psif(\hmu_n) - \Psif(\mu^\star)$.  Then,
\begin{align}\label{eq:rate1:sto}
\PE [\Delta_n] &\leq \PE \left[ \int_\Tset \bmuf[\hmu_n](\rmd \hmu_n - \rmd \mu^\star) \right] \\
& =  \PE \left[ \int_\Tset \PE \left[ \bmufk[\hmu_n] | \mathcal F_n \right] (\rmd \hmu_n - \rmd \mu^\star) \right] \nonumber\\
& = \PE \left[ \int_\Tset \bmufk[\hmu_n](\rmd \hmu_n - \rmd \mu^\star) \right]  \nonumber \\
&= \PE \left[\int_\Tset (\hmu_n(\bmufk[\hmu_n]) - \bmufk[\hmu_n]) \rmd \mu^\star \right] \eqsp. \nonumber
\end{align}
By adapting the proof of \Cref{thm:admiss}, we deduce
$$
\PE \left[\Delta_n \right] \leq \PE \left[\frac{1}{(- \log \GammaAlpha)'(\hmu_n(\bmufk[\hmu_n]))} \int_\Tset \left[ \log \GammaAlpha(\bmufk[\hmu_n]) - \log \GammaAlpha(\hmu_n(\bmufk[\hmu_n])) \right] \rmd \mu^\star \right] \eqsp.
$$
In the particular case of the Entropic Mirror Descent (for which $\GammaAlpha(v) = e^{-\eta v}$) we obtain that $-(\log \GammaAlpha)' = \eta$, that is $\cteinf = \frac{1}{\eta}$ and by following the proof of \Cref{thm:admiss} we have
\begin{align}\label{eq:interEtaSto}
\PE\left[\Delta_n\right] &\leq \frac{1}{\eta} \PE\left[ \int_\Tset \log \left(\frac{\rmd \hmu_{n+1}}{\rmd \hmu_n} \right)\rmd \mu^\star  + \frac{L}{2} \Var_{\hmu_n} \left( \bmufk[\hmu_n]\right) \ctesup \right] \eqsp.
\end{align}
By assumption on $\inf_{v \in \Domain} \GammaAlpha(v)$ and on $\sup_{v \in \Domain} \GammaAlpha(v)$, we have that $L_{\alpha, 4}< \infty$ and combining with \Cref{lem:mono:refined:sto}, we obtain
\begin{align*}
\PE\left[\Delta_n\right] &\leq \frac{1}{\eta} \PE\left[ \int_\Tset \log \left(\frac{\rmd \hmu_{n+1}}{\rmd \hmu_n} \right)\rmd \mu^\star  + \frac{L \ctesup}{\ctemono} \lrb{ \Psif(\hmu_n) - \Psif(\hmu_{n+1}) } \right] \\
 &\quad + \frac{L \ctesup L_{\alpha,4}}{\eta \ctemono}  \frac{\sigma}{\sqrt{M}}  \eqsp.
\end{align*}
As the r.h.s involves two telescopic sums, we deduce
\begin{align*}
\frac{1}{N} \sum_{n=1}^{N} \PE\left[\Delta_n\right] \leq \frac{1}{\eta N} \lrb{ KL\couple[\mu^\star][\hmu_1] + L\frac{ \ctesup}{\ctemono} \Delta_1 } +  \frac{L \ctesup L_{\alpha,4}}{\eta \ctemono}  \frac{\sigma}{\sqrt{M}} \eqsp,
\end{align*}
and we conclude using the convexity of the mapping $\mu \mapsto \Psif(\mu)$.
\end{proof}
With all these elements in hand, we can now prove \Cref{thm:admiss:spec:sto}.
\begin{proof}[Proof of \Cref{thm:admiss:spec:sto}]
The proof follows from a straightforward adaptation of the proof of \Cref{thm:admiss:spec} for the Entropic Mirror Descent (we replace $\binfty$ by $\hbinfty$) combined with \Cref{thm:admiss:spec:sto}.
\end{proof}

\subsection{Proof of \Cref{thm:emd:sto}}
\label{subsec:CVstoMD}

\begin{proof}[Proof of \Cref{thm:emd:sto}]The proof of \Cref{thm:emd:sto} can be adapted from the proof of \cite[Section 2.3]{doi:10.1137/070704277}. We consider the case $\cte = 0$ for simplicity. Note that the case $\cte \neq 0$ unfolds similarly by replacing $\bmuf[\hmu_n]$ by $\bmuf[\hmu_n] + \cte$ everywhere in the proof below. Let $n \in \nstar$ and set $\Delta_n = \Psif(\hmu_n) - \Psif(\mu^\star)$. The convexity of $\falpha$ implies that
\begin{align*}
\Delta_n &\leq \int_\Tset \bmuf[\hmu_n] (\rmd \hmu_n-\rmd \mu^\star) \eqsp.
\end{align*}
Now taking the expectation, we obtain that
\begin{align}\label{eq:MDSto}
\PE[\Delta_n] &\leq \PE \left[ \int_\Tset  \bmuf[\hmu_n] (\rmd \hmu_n-\rmd \mu^\star)\right] \\
&= \PE \left[ \int_\Tset \mathbb E [  \bmufk[\hmu_n] \mid \mathcal F_n ] (\rmd \hmu_n-\rmd \mu^\star)\right] \nonumber \\
& = \PE \left[ \int_\Tset \bmufk[\hmu_n]  (\rmd \hmu_n-\rmd \mu^\star)\right] \nonumber \eqsp.
\end{align}
In addition, using that $\frac{\rmd \hmu_{n+1}}{\rmd \hmu_n} \propto e^{-\eta_n\bmufk[\hmu_n]}$ and noting that the integral of any constant w.r.t $\hmu_n-\mu^\star$ is null, we deduce
\begin{align*}
\int_\Tset  \bmufk[\hmu_n] (\rmd \hmu_n- \rmd \mu^\star) &=\frac{1}{\eta_n} \int_\Tset \log\lr{\frac{\rmd \hmu_n}{ \rmd \hmu_{n+1}}} (\rmd \hmu_n-\rmd \mu^\star)\\
&=\frac{1}{\eta_n} \int_\Tset \log\lr{\frac{\rmd \hmu_n}{\rmd \hmu_{n+1}}} \rmd \hmu_n-\frac{1}{\eta_n} \int_\Tset \log\lr{\frac{\rmd \hmu_n}{\rmd \hmu_{n+1}}} \rmd \mu^\star\\
&=\frac{1}{\eta_n}\lrb{ \int_\Tset \log\lr{\frac{\rmd \hmu_n}{\rmd \hmu_{n+1}}} (\rmd \hmu_n-\rmd \hmu_{n+1})-KL(\hmu_{n+1}||\hmu_n)} \\
&\qquad + \frac{1}{\eta_n}\lrb{ KL(\mu^\star||\hmu_n)-KL(\mu^\star||\hmu_{n+1})}
\end{align*}
Let us first consider the term inside the first brackets. We have that
\begin{align*}
\int_\Tset \log\lr{\frac{\rmd \hmu_n}{\rmd \hmu_{n+1}}} (\rmd \hmu_n-\rmd \hmu_{n+1})
&=\eta_n\int_\Tset  \bmufk[\hmu_n] (\rmd \hmu_n-\rmd \hmu_{n+1}) \\
&\leq \eta_n \cteLipSto[n] \tv{\hmu_n-\hmu_{n+1}} \eqsp,
\end{align*}
where we have set
$$
\cteLipSto[n] = \frac{1}{M} \sum_{m=1}^{M} \sup_{\theta \in \Tset} \frac{k(\theta,Y_{m,n+1})}{\hmu_n k(Y_{m,n+1})} \left| \falpha ' \lr{\frac{\hmu_n k(Y_{m,n+1})}{p(Y_{m,n+1})}} \right|
$$
and where we have used that $\frac{\rmd \hmu_{n+1}}{\rmd \hmu_n} \propto e^{-\eta_n\bmufk[\hmu_n]}$ and that the integral of any constant w.r.t $\hmu_n-\hmu_{n+1}$ is null. Moreover, Pinsker's inequality yields
$$
-KL(\hmu_{n+1}|| \hmu_n) \leq -\frac12 \tv{\hmu_n-\hmu_{n+1}}^2 \eqsp.
$$
Now combining with the fact that $\eta_n\cteLipSto[n] a -a^2/2 \leq (\eta_n \cteLipSto[n])^2/2$ which is valid for all $a\geq0$, we get:
\begin{align}\label{eq:StoBound}
\frac{1}{\eta_n}\lrb{ \int_\Tset \log\lr{\frac{\rmd \hmu_n}{\rmd \hmu_{n+1}}} (\rmd \hmu_n-\rmd \hmu_{n+1})-KL(\hmu_{n+1} || \hmu_n)} \leq  \frac{\eta_n  \cteLipSto[n]^2}{2}.
\end{align}
Furthermore, using Jensen's inequality and \eqref{bound_unif}, we have that
\begin{align*}
  \PE \lrb{\cteLipSto[n]^2} & \leq \PE \lrb{ \frac{1}{M} \sum_{m=1}^M \left(\sup_{\theta \in \Tset} \frac{k(\theta,Y_{m,n+1})}{\hmu_n k(Y_{m,n+1})} \left| \falpha ' \lr{\frac{\hmu_n k(Y_{m,n+1})}{p(Y_{m,n+1})}} \right|\right)^2 } \\
  & \leq \cteLipStoInf^2
\end{align*}
and as a consequence, we obtain from \eqref{eq:MDSto} and \eqref{eq:StoBound} that
\begin{align*}
\eta_n \PE [\Delta_n] \leq \eta_n^2 \cteLipStoInf^2 /2+ \mathbb E [(KL(\mu^\star|| \hmu_n)-KL(\mu^\star|| \hmu_{n+1})) ] \eqsp.
\end{align*}
Finally, as we recognize a telescoping sum in the right-hand side, we have
$$
\sum_{n=1}^N \eta_n \PE [\Delta_n] \leq \sum_{n=1}^N \eta_n^2 \cteLipStoInf^2 /2+ KL(\mu^\star || \hmu_{1}) \eqsp
$$
that is we have, by convexity of the mapping $\mu \mapsto \Psif(\mu)$,
\begin{align}\label{eq:interEtaSto2}
\PE\lrb{\Psif\lr{\sum_{n=1}^N w_n \hmu_n}-\Psif(\mu^\star)}\leq \frac{\cteLipStoInf^2 \sum_{n=1}^N \eta_n^2/2}{\sum_{n=1}^{N} \eta_n} + \frac{KL(\mu^\star || \hmu_1)}{\sum_{n=1}^{N} \eta_n} \eqsp.
\end{align}
Then,
\begin{itemize}
  \item setting $\eta_n = \eta_0/\sqrt{n}$ for all $n \geq 1$ in \eqref{eq:interEtaSto2} yields
$$
\PE\lrb{ \Psif\lr{\sum_{n=1}^N w_n \hmu_n}-\Psif(\mu^\star)} \leq \frac{(1+\log(N)) \cteLipStoInf^2 \eta_0^2/2 + KL(\mu^\star || \hmu_1)}{\eta_0 \sqrt{N}}  \eqsp.
$$
  \item setting $\eta_n = \eta_0/\sqrt{N}$ for all $n = 1 \ldots N$ \eqref{eq:interEtaSto2} yields
$$
   \PE\lrb{ \Psif\lr{\frac{1}{N} \sum_{n=1}^N  \hmu_n}-\Psif(\mu^\star)} \leq \frac{\cteLipStoInf^2 \eta_0^2/2 + KL(\mu^\star || \hmu_1)}{\eta_0 \sqrt{N}}
$$
Furthermore, the r.h.s is minimal for $\eta_0 = \cteLipStoInf^{-1} \sqrt{2 KL(\mu^\star|| \hmu_1)}$ that is for $\eta_n = \cteLipStoInf^{-1} \sqrt{\frac{2 KL(\mu^\star|| \hmu_1)}{N}}$ for all $n = 1 \ldots N$.
\end{itemize}

\end{proof}


\subsection{\Cref{ex:GaussianMixtureModels} and Condition \eqref{bound_unif}}
\label{sec:BoundUnif}

\begin{proof}[Proof that Condition \eqref{bound_unif} is satisfied in \Cref{ex:GaussianMixtureModels}] \ \\

We have $k_h(\theta,y) = \frac{e^{-\| y - \theta \|^2 / (2h^2)}}{(2 \pi h^2)^{d/2}}$ and $p(y) = 0.5 \frac{e^{-\| y - \theta^\star_1 \|^2 / 2}}{(2 \pi)^{d/2}}+ 0.5 \frac{e^{-\| y - \theta^\star_2 \|^2 / 2}}{(2 \pi)^{d/2}}$ for all $\theta \in \Tset$ and all $y \in \Yset$. Since we have chosen $\alpha = 1$, we have $\falpha '(u) = \log(u)$ for all $u > 0$ and we are interested in the following quantity
$$
\cteLipStoInf[1]^2 \eqdef \sup_{\mu \in \meas{1}(\Tset)} \int_\Yset \sup_{\theta, \theta ' \in \Tset} \frac{k_h(\theta,y)^2}{k_h(\theta ', y)} \left| \log\lr{ \frac{\mu k_h(y)}{p(y)}} \right|^2 \nu(\rmd y) \eqsp.
$$
Recall that by assumption $\Tset = \mathcal{B}(0,r)$. Then, for all $\theta, \theta' \in \Tset$ and for all $y \in \Yset$, we can write
\begin{align*}
  \frac{k_h(\theta, y)}{k_h(\theta ', y)} &= e^{\frac{-\| y - \theta \|^2 + \| y - \theta ' \|^2}{2h^2}} = e^{\frac{2 <y, \theta - \theta '> - \| \theta \|^2 + \| \theta '\|^2}{2h^2}} \\
  & \leq e^{\frac{2 |<y,\theta - \theta'>| + \| \theta\|^2 + \|\theta '\|^2 }{2h^2}} \\
  & \leq  e^{\frac{\|y \| \| \theta - \theta'  \| + r^2}{h^2}} \\
  & \leq e^{\frac{\|y \| 2 r + r^2}{h^2}}
\end{align*}
Furthermore, we also have for all $y \in \Yset$
\begin{align*}
& e^{-\sup_{\theta \in \Tset} \frac{\| y - \theta \|^2}{2h^2}}\leq (2 \pi h^2)^{d/2} \mu k_h(y)  \leq 1 \\
& e^{-\max_{i \in \lrcb{1,2}} \frac{\| y - \theta_i^\star \|^2}{2}}\leq (2 \pi )^{d/2} p(y) \leq 1
\end{align*}
and we can deduce for all $\mu \in \meas{1}(\Tset)$ and all $y \in \Yset$
\begin{align*}
\left| \log\lr{ \frac{\mu k_h(y)}{p(y)}} \right| & \leq \sup_{\theta \in \Tset} \frac{\| y - \theta \|^2}{2h^2}  + \max_{i \in \lrcb{1,2}} \frac{\| y - \theta_i^\star \|^2}{2} + d |\log h| \\
& \leq \frac{(\| y \| + r)^2}{2}\lrb{\frac{1}{h^2} + 1} + d |\log h| \eqsp.
\end{align*}
Consequently, we have
$$
 \cteLipStoInf[1]^2 \leq \int_\Yset \frac{ e^{\frac{\|y \| 2 r + r^2}{h^2}}}{(2 \pi h^2)^{d/2}} \underbrace{\sup_{\theta \in \Tset} e^{-\| y - \theta \|^2 / (2h^2)}}_{\leq e^{-(\|y\| - r)_+^2/(2h^2)}} \lr{\frac{(\| y \| + r)^2}{2}\lrb{\frac{1}{h^2} + 1} + d |\log h|}^2 \nu(\rmd y)
$$
that is $\cteLipStoInf[1] < \infty$.
\end{proof}

\section{}

\subsection{\Cref{lem:discretize} : statement and proof}

Recall that $Y_1, Y_2, \ldots$ are i.i.d random variables with common density $\mu k$ w.r.t $\nu$, defined on the same probability space $(\Omega,\mcf,\PP)$ and we denote by $\PE$ the associated expectation operator. Here, $\GammaAlpha$ is chosen as $\GammaAlpha(v) = [ \left(\alpha -1 \right)v + 1]^{\eta/(1-\alpha)}$.

\begin{lemma}\label{lem:discretize}
Assume \ref{hyp:positive}. Let $\alpha \in \Rset \setminus \lrcb{1}$, $\eta > 0$ and $\cte$ be such that $(\alpha-1)\cte \geq 0$. Let $\mu \in \meas{1}(\Tset)$ be such that  $\mu(|\bmuf|) < \infty$ and
\begin{equation}
\label{eq:discretizeMinus0}
\int_\Tset \mu(\rmd \theta)  \PE\lrb{\lrcb{\frac{k(\theta,Y_1)}{\mu k(Y_1)} \lr{\frac{\mu k(Y_1)}{p(Y_1)}}^{\alpha-1} + (\alpha-1)\cte}^{\frac{\eta}{1-\alpha}}} <\infty\eqsp.
\end{equation}
Then,
\begin{equation} \label{eq:discretize0}
\llim_{M \to \infty} \mu( \GammaAlpha(\bmufk+ \cte)) = \mu(\GammaAlpha(\bmuf+ \cte))\,, \quad \PP-\as
\end{equation}
\end{lemma}

\begin{proof}
Set $g(\theta, y) = \frac{k(\theta,y) }{\mu k(y)} ( \frac{\mu k(y)}{p(y)} )^{\alpha-1} + (\alpha-1)\cte$, $\phi = \frac{\eta}{1-\alpha}$ and $h(u) = (\alpha-1)u + (\alpha-1)\cte + 1$. Note that $\PE[g(\theta, Y_1)] = h(\bmuf(\theta))$ and $h^{\phi} = \GammaAlpha$.

\begin{enumerateList}

\item We start with the case $\phi \notin [0,1]$. Our goal is to apply \Cref{lem:gen:tcd}, which is a generalized version of the Dominated Convergence Theorem. To do so, first note that $h(\bmufk(\theta) )^\phi$ is positive and combining with the convexity of the mapping $u \mapsto u^\phi$, we have for all $M \in \nset^\star$ and for all $\theta \in \Tset$,
\begin{align}\label{eq:bound-TCD}
0 \leq h(\bmufk(\theta))^\phi \leq M^{-1} \sum_{m = 1}^M [g(\theta, Y_m)]^\phi \eqsp.
\end{align}
Since $\mu(|\bmuf|) < \infty$, the LLN for $\mu$-almost all $\theta \in \Tset$ yields
\begin{align}\label{eq:bmuk:lln}
\llim_{M \to \infty} \bmufk(\theta) = \bmuf(\theta) \eqsp.
\end{align}
Now applying successively (a) the LLN for $\mu$-almost all $\theta\in \Tset$ (as stated in \Cref{lem:lln:integrated}), which is valid under \eqref{eq:discretizeMinus0}, (b) Fubini's Theorem and (c) again the LLN
 \begin{multline}
\int_\Tset \mu(\rmd \theta) \lim_{M \to \infty} M^{-1} \sum_{m = 1}^M \{g(\theta, Y_m)\}^\phi \stackrel{(a)}{=} \int_\Tset \mu(\rmd \theta) \PE\lrb{\{g(\theta, Y_1)\}^\phi} \\ \stackrel{(b)}{=}  \PE\lrb{\int_\Tset \mu(\rmd \theta) [g(\theta, Y_1)]^\phi} \stackrel{(c)}{=} \lim_{M \to \infty} \int_\Tset \mu(\rmd \theta) M^{-1} \sum_{m = 1}^M [g(\theta, Y_m)]^\phi
\end{multline}
That is
$$
\mu \left(\lim_{M \to \infty} M^{-1} \sum_{m = 1}^M \{g(\cdot, Y_m)\}^\phi \right) = \lim_{M \to \infty} \mu \left( M^{-1} \sum_{m = 1}^M [g(\cdot, Y_m)]^\phi \right) < \infty
$$
Combining with \eqref{eq:bound-TCD} and \eqref{eq:bmuk:lln}, we apply \Cref{lem:gen:tcd} and obtain
$$
\mu \lr{  h(\bmuf)^\phi} = \mu \lr{ \lim_{M \to \infty}  h(\bmufk)^\phi} = \lim_{M \to \infty} \mu( h(\bmufk)^\phi) \eqsp,
$$
that is
$$
\mu \lr{  \GammaAlpha(\bmuf + \cte)} = \lim_{M \to \infty} \mu( \GammaAlpha(\bmufk + \cte)) \eqsp.
$$

\item We now turn to the case $\phi \in (0,1]$. Let $M' > 0$. Since
$$
\int_\Tset \mu(\rmd \theta) \lr{M^{-1}\sum_{m=1}^{M} g(\theta,Y_m)  \indiacc{g(\theta,Y_m)\leq M'}}^\phi\leq \mu( h(\bmufk )^\phi)\eqsp,
$$
the LLN for $\mu$-almost all $\theta\in \Tset$ (\Cref{lem:lln:integrated}) and the Dominated Convergence Theorem yields
\begin{equation}
\label{eq:discretize3}
 \int_\Tset \mu(\rmd \theta) \lr{\PE[g(\theta,Y_1)  \indiacc{g(\theta,Y_1)\leq M'}]}^\phi\leq \lliminf_{M \to \infty} \mu( h(\bmufk)^\phi) \eqsp.
\end{equation}
Using now $(u+v)^\phi \leq u^\phi+v^\phi$ and then Jensen's inequality for the concave mapping $u \mapsto u^\phi$,
\begin{multline*}
\mu( h(\bmufk)^\phi) \leq \int_\Tset \mu(\rmd \theta) \lr{M^{-1}\sum_{m=1}^{M} g(\theta,Y_m)  \indiacc{g(\theta,Y_m)\leq M'}}^\phi\\
 +\lr{ \int_\Tset \mu(\rmd \theta) M^{-1}\sum_{m=1}^{M} g(\theta,Y_m)  \indiacc{g(\theta,Y_m)> M'}}^\phi
\end{multline*}
By invoking the LLN for $\mu$-almost all $\theta\in \Tset$ (\Cref{lem:lln:integrated}) and the Dominated Convergence Theorem for the first term of the rhs and the LLN combined with Fubini for the second term, we get
\begin{multline*}
\llimsup_{M \to \infty} \mu( h(\bmufk)^\phi) \leq   \int_\Tset \mu(\rmd \theta) \lr{\PE[g(\theta,Y_1)  \indiacc{g(\theta,Y_1)\leq M'}]}^\phi\\
 +  \lr{\int_\Tset \mu(\rmd \theta) \PE[g(\theta,Y_1)  \indiacc{g(\theta,Y_1)> M'}]}^\phi
\end{multline*}
Letting $M'$ go to infinity both in this inequality and in \eqref{eq:discretize3} completes the proof of \eqref{eq:discretize0}.
\end{enumerateList}
\end{proof}

\subsection{General Dominated Convergence Theorem}

We state and prove a generalized version of the Dominated Convergence Theorem, adapted from \cite[Theorem 19]{royden2010real}. We provide here a full proof for the sake of completeness.
\begin{lemma}[General Dominated Convergence Theorem] \label{lem:gen:tcd}
Let $\zeta \in \meas{1}(\Tset)$. Assume there exist $(a_M)$, $(b_M), (c_M)$  three sequences of $(\Tsigma,{\mathcal B}(\rset))$-measurable functions such that the limits $\lim_{M \to \infty} a_M(\theta)$, $\lim_{M \to \infty} b_M(\theta)$, $\lim_{M \to \infty} c_M(\theta)$ exist for $\zeta$-almost all $\theta \in \Tset$ and
$$
\zeta|\lim_{M \to \infty} a_M| +\zeta|\lim_{M \to \infty} c_M|<\infty \eqsp.
$$
Assume moreover that for all $M \in \nset^\star$ and for $\zeta$-almost all $\theta \in \Tset$
\begin{align*}
a_M(\theta) \leq b_M(\theta) \leq c_M(\theta)
\end{align*}
and
\begin{align}\label{eq:gen:fatou}
\zeta( \lim_{M \to\infty} a_M) =  \lim_{M \to\infty} \zeta(a_M)  \\
\zeta( \lim_{M \to\infty} c_M) =  \lim_{M \to\infty} \zeta(c_M) \label{eq:gen:fatou2} \eqsp.
\end{align}
Then,
\begin{align*}
\zeta(\lim_{M \to\infty} b_M) = \lim_{M \to\infty} \zeta(b_M) \eqsp.
\end{align*}
\end{lemma}

\begin{proof} We apply Fatou's Lemma combined with \eqref{eq:gen:fatou} and \eqref{eq:gen:fatou2} to the two non-negative, $(\Tsigma,{\mathcal B}(\rset))$-measurable functions $\theta \mapsto b_M(\theta) - a_M(\theta)$ and $\theta \mapsto c_M(\theta) - b_M(\theta)$ and we obtain
\begin{align*}
\zeta (\liminf_{M \to\infty} b_M) &\leq \liminf_{M \to\infty} \zeta(b_M) \\
\zeta (\liminf_{M \to\infty} -b_M) &\leq \liminf_{M \to\infty} \zeta(-b_M)
\end{align*}
which proves the lemma, as $\liminf_{M \to \infty} b_M(\theta)=\limsup_{M \to \infty} b_M(\theta)$ for $\zeta$-almost all $\theta \in \Tset$.
\end{proof}

\subsection{Integrated Law of Large Numbers}

Let $Y_1,Y_2,\ldots$ be i.i.d. random variables on the same probability space $(\Omega,\mcf,\PP)$ and let $f$ be a non-negative real-valued $(\Tsigma \otimes \mcf,{\mathcal B}(\rset_{\geq 0}))$-measurable function. We are interested in showing
\begin{align} \label{eq:LLNintegre}
\int_\Tset \zeta(\rmd \theta) \lim_{M \to\infty}  M^{-1}\sum_{m=1}^M f(\theta,Y_m)=\int_\Tset \zeta(\rmd \theta) \PE[f(\theta,Y_1)]
\end{align}
for $\zeta \in \meas{1}(\Tset)$ satisfying $\int_\Tset \zeta(\rmd \theta) \PE[f(\theta,Y_1)]<\infty$. While this result follows easily if we can show that
\begin{equation} \label{eq:LLNunif}
\PP \lr{\forall \theta \in \Tset, \eqsp \lim_{M \to\infty}  M^{-1}\sum_{m=1}^M f(\theta,Y_m)= \PE[f(\theta,Y_1)]}=1
\end{equation}
unfortunately the LLN only yields
$$
\PP \lr{\lim_{M \to\infty}  M^{-1}\sum_{m=1}^M f(\theta,Y_m)= \PE[f(\theta,Y_1)]}=1
$$
for $\zeta$-almost all $\theta\in \Tset$. The following lemma allows to show \eqref{eq:LLNintegre} without resorting to the much stronger identity \eqref{eq:LLNunif}.
\begin{lemma} \label{lem:lln:integrated}
Let $\zeta \in \meas{1}(\Tset)$ and assume that $\int_\Tset \zeta(\rmd \theta) \PE[f(\theta,Y_1)]<\infty$. Then, $\PP-\as$
\begin{align*}
\int_\Tset \zeta(\rmd \theta) \lim_{M \to\infty}  M^{-1}\sum_{m=1}^M f(\theta,Y_m)=\int_\Tset \zeta(\rmd \theta) \PE[f(\theta,Y_1)]\;.
\end{align*}
\end{lemma}

\begin{proof}
Set
$$
B=\set{(\theta,\omega) \in \Tset \times\Omega}{\llim_{M \to \infty}  M^{-1}\sum_{m=1}^M f(\theta,Y_m(\omega))= \PE[f(\theta,Y_1)]}\eqsp,
$$
Let $\gamma_0: (\theta,\omega) \mapsto \indi{B^c}(\theta,\omega)$ and $\gamma_1=1-\gamma_0$. According to the Fubini Theorem and the LLN for $M^{-1}\sum_{m=1}^M f(\theta,Y_m)$ where $\theta$ is such that $\PE[f(\theta,Y_1)] < \infty$ (which is satisfied for $\zeta$-almost all $\theta \in \Tset$ by assumption),
$$
\PE\lrb{\int_\Tset \zeta(\rmd \theta)  \gamma_0(\theta, \cdot)}=\int_\Tset \zeta(\rmd \theta) \PE\lrb{  \gamma_0(\theta, \cdot)}=0 \eqsp.
$$
Therefore, $\int_\Tset \zeta(\rmd \theta)  \gamma_0(\theta, \cdot)$ is $\PP-\as$ null that is, there exists $\Omega_1$ such that $\PP(\Omega_1)=1$ and for all $\omega\in \Omega_1$, $A \mapsto \int_A \zeta(\rmd \theta)  \gamma_0(\theta, \omega)$ is the null-measure on $(\Tset,\Tsigma)$, which in turn implies that the measures $\zeta$ and $A \mapsto \int_A \zeta(\rmd \theta)\gamma_1(\theta,\omega)$ coincide. The latter property implies for all $\omega\in \Omega_1$,
\begin{align*}
\int_\Tset \zeta(\rmd \theta) \PE[f(\theta,Y_1)] &=\int_\Tset \zeta(\rmd \theta) \PE[f(\theta,Y_1)] \gamma_1(\theta,\omega)\\
&=\int_\Tset \zeta(\rmd \theta)\lrb{ \llim_{M \to \infty}  M^{-1}\sum_{m=1}^M f(\theta,Y_m(\omega))} \gamma_1(\theta,\omega)\\
&=\int_\Tset \zeta(\rmd \theta) {\llim_{M \to \infty}  M^{-1}\sum_{m=1}^M f(\theta,Y_m(\omega))}\eqsp.
\end{align*}
\end{proof}

\subsection{Proof of \Cref{prop:discretize}}
\label{subsec:proof:discretize}

\begin{proof}[Proof of \Cref{prop:discretize}] Recall that we have taken $\GammaAlpha(v) = [ \left(\alpha -1 \right)v + 1]^{\eta/(1-\alpha)}$. For the sake of readability, we only treat the case $\cte = 0$ in the proof of \Cref{prop:discretize}. Note that the case $\cte \neq 0$ unfolds similarly by replacing $\bmuf$ by $\bmuf + \cte$ everywhere in the proof below.

A first remark is that $\Psif(\mu) < \infty$ implies $\mu(|\bmuf|) < \infty$. This comes from the fact that for all $\alpha \in \rset$ and for all $u \in \rset_{>0}$, $u \falpha'(u) = \alpha \falpha(u) + (u-1)$ and we can write
\begin{align*}
\mu(|\bmuf[\mu]|) &\leq |\alpha| \int_\Yset \left|\falpha\lr{\frac{\mu k(y)}{p(y)}}\right| p(y)\nu(\rmd y) + \int_\Yset p(y) \nu(\rmd y) + 1 \eqsp.
\end{align*}
Under \ref{hyp:positive}, we have $\int_\Yset p(y) \nu(\rmd y) < \infty$, which settles the case $\alpha = 0$. As for the case $\alpha \in \rset \setminus \lrcb{0}$, we obtain from \Cref{lemma:bound} that the r.h.s is finite if and only if $\Psif(\mu)$ is finite, which is implied by the assumption $\Psif(\mu) < \infty$.

By the triangular inequality, for all $M \in \nset^\star$, for all $\theta \in \Tset$,
\begin{align*}
\left| \frac{\GammaAlpha(\bmufk(\theta))}{\mu(\GammaAlpha(\bmufk))}-\frac{\GammaAlpha(\bmuf(\theta))}{\mu(\GammaAlpha(\bmuf))}\right| \leq & \ \frac{\GammaAlpha(\bmufk(\theta))}{\mu(\GammaAlpha(\bmufk))}\left|1-\frac{\mu(\GammaAlpha(\bmufk))}{\mu(\GammaAlpha(\bmuf))} \right|\\
&+ \frac{|\GammaAlpha(\bmufk(\theta))-\GammaAlpha(\bmuf(\theta))|}{\mu(\GammaAlpha(\bmuf))}
\end{align*}
Thus,
\begin{align*}
\tv{\iterationK (\mu) -\iteration (\mu)}&=\mu \lr{\left| \frac{\GammaAlpha(\bmufk)}{\mu(\GammaAlpha(\bmufk))}-\frac{\GammaAlpha(\bmuf)}{\mu(\GammaAlpha(\bmuf))}\right|}\\
&\leq \left|1-\frac{\mu(\GammaAlpha(\bmufk))}{\mu(\GammaAlpha(\bmuf))} \right|+\frac{\mu(|\GammaAlpha(\bmufk)-\GammaAlpha(\bmuf)|)}{\mu(\GammaAlpha(\bmuf))}
\end{align*}
For the first term of the rhs, \Cref{lem:discretize} yields
\begin{align}\label{eq:TV:rhs1}
\lim_{M \to \infty} \left|1-\frac{\mu(\GammaAlpha(\bmufk))}{\mu(\GammaAlpha(\bmuf))} \right| = 0
\end{align}
As for the second term of the rhs, first note that for all $M \in \nset^\star$, for all $\theta \in \Tset$
\begin{align}\label{eq:TV:TCD1}
0 \leq |\GammaAlpha(\bmufk(\theta))-\GammaAlpha(\bmuf(\theta))| \leq \GammaAlpha(\bmufk(\theta)) + \GammaAlpha(\bmuf(\theta)) \eqsp,
\end{align}
and since $\mu(\GammaAlpha(\bmuf)) < \infty$ by assumption, the LLN for $\mu$-almost all $\theta \in \Tset$ yields
\begin{align}\label{eq:TV:TCD2}
\lim_{M \to \infty} \GammaAlpha(\bmufk(\theta)) = \GammaAlpha(\bmuf(\theta)) \eqsp.
\end{align}
Furthermore, since $\mu(\GammaAlpha(\bmuf)) < \infty$, \Cref{lem:discretize} and \eqref{eq:TV:TCD2} imply
\begin{align*}
\lim_{M \to \infty} \mu \left[ \GammaAlpha(\bmufk) + \GammaAlpha(\bmuf) \right] = \mu \left[ \lim_{M \to \infty} \left( \GammaAlpha(\bmufk) + \GammaAlpha(\bmuf) \right) \right] < \infty
\end{align*}
Combining with \eqref{eq:TV:TCD1} and \eqref{eq:TV:TCD2}, we apply \Cref{lem:gen:tcd} and obtain
\begin{align*}
\lim_{M \to \infty} \frac{\mu(|\GammaAlpha(\bmufk)-\GammaAlpha(\bmuf)|)}{\mu(\GammaAlpha(\bmuf))} = 0
\end{align*}
which, along with \eqref{eq:TV:rhs1}, finishes the proof.
\end{proof}

\section{Additional remarks}
\label{sec:addRes}

%
\begin{rem}[Assumption $\Psif(\mu_1)<\infty$  in \Cref{thm:admiss:spec}]\label{rep:hyp:admiss:spec}
The assumption $\Psif(\mu_1)<\infty$ can be discarded in \Cref{thm:admiss:spec} for all $\alpha \neq 0$. Indeed, for all $\alpha \in \rset$ and for all $u >0$, we have that $u \falpha'(u) = \alpha \falpha(u) + u-1$ and thus we can write
\begin{align*}
\mu_1(\bmuf[\mu_1]) & = \alpha \int_\Yset \falpha\lr{\frac{\mu_1 k(y)}{p(y)}} p(y)\nu(\rmd y) + \int_\Yset p(y) \nu(\rmd y) + 1 \eqsp.
\end{align*}
Under \ref{hyp:positive}, it holds that $\int_\Yset p(y) \nu(\rmd y) < \infty$. Combined with the fact that we have assumed that $\binfty < \infty$ in \Cref{thm:admiss:spec}, we obtain that $\Psif(\mu_1) < \infty$.
\end{rem}

\begin{rem}[Renyi-bound]\label{rem:RenyiBound}
For any variational density $q$, the ELBO and the Renyi-bound (its extension to the $\alpha$-divergence family) are respectively defined by
\begin{align*}
\mathcal{L}_{1}(q; \data) & \eqdef \int_\Yset \log \left( \frac{p(y, \data)}{q(y)} \right) q(y) \nu(\rmd y ) \\
\mathcal{L}_\alpha(q; \data) & \eqdef \frac{1}{1-\alpha} \log \left( \int_\Yset \left( \frac{p(y, \data)}{q(y)} \right)^{1-\alpha} q(y) \nu(\rmd y ) \right) \eqsp.
\end{align*}
Then, in our particular framework, we can write
\begin{align}\label{eq:renyibound}
\mathcal{L}_{1}(\hmu_n k, \data) &= -\sum_{j=1}^J \lbd[j]{n} \bmuf[\hmu_n](\theta_j) \\
\mathcal{L}_\alpha(\hmu_n k, \data) &= \frac{1}{1-\alpha} \log \left( (\alpha-1)\sum_{j=1}^J \lbd[j]{n} \bmuf[\hmu_n](\theta_j) + 1 \right) \eqsp. \nonumber
\end{align}
\end{rem}


\end{document}